\documentclass[12pt,reqno]{amsart}
\usepackage[margin=1in]{geometry}
\usepackage{graphicx}
\usepackage[numbers,sort]{natbib}
\usepackage{amsmath,amsthm,amsfonts,mathrsfs,amssymb,float,color}
\usepackage{graphicx}
\usepackage{textcomp}
\usepackage{verbatim}
\usepackage{ upgreek }
\usepackage[T1]{fontenc}
\usepackage[utf8]{inputenc}
\usepackage{hyperref}
\hypersetup{colorlinks=true,citecolor=blue,linkcolor=blue}
\usepackage{epsfig,subfigure,fancybox,balance}
\usepackage{diagbox}
\usepackage{youngtab}

\newcommand{\beq}[1]{\begin{equation} \label{#1}}
\newcommand{\eeq}{\end{equation}}
\newcommand{\bea}{\bed\begin{array}{rl}}
\newcommand{\eea}{\end{array}\eed}
\newcommand{\bed}{\begin{displaymath}}
\newcommand{\eed}{\end{displaymath}}
\newcommand{\barray}{\begin{array}{ll}}
\newcommand{\earray}{\end{array}}
\newcommand{\disp}{\displaystyle}
\newcommand{\ad}{&\!\disp}
\newcommand{\aad}{&\disp}
\newcommand{\al}{\alpha}

\newcommand{\la}{\lambda}
\newcommand{\La}{\Lambda}
\newcommand{\sg}{\sigma}
\newcommand{\Sg}{\Sigma}

\newcommand{\Ga}{\Gamma}
\newcommand{\dl}{\delta}
\newcommand{\Dl}{\Delta}
\newcommand{\cd}{(\cdot)}

\def\phi{\varphi}
\def\indi{{\bf 1}}
\def\half{\frac{1}{2}}
\newcommand{\vsg}{\varsigma}

\newcommand{\CF}{{\mathcal F}}

\newcommand{\CU}{\mathcal{U}}

\newcommand{\CO}{\mathcal{O}}

\newcommand{\CL}{\mathcal{L}}
\newcommand{\CP}{\mathcal{P}}

\newcommand{\CR}{\mathcal{R}}

\newcommand{\CN}{\mathcal{N}}

\newcommand{\CQ}{\mathcal{Q}}
\newcommand{\EE}{{\mathbb E}}
\newcommand{\PP}{{\mathbb P}}

\newcommand{\NN}{{\mathbb N}}
\newcommand{\rr}{{\mathbb R}}
\newcommand{\QQ}{{\mathbb Q}}
\newcommand{\fL}{\mathfrak{L}}

\newcommand{\wdt}{\widetilde}
\newcommand{\wdh}{\widehat}

\newcommand{\qv}[1]{\langle #1 \rangle}
\newcommand{\bqv}[1]{\big\langle #1 \big\rangle}
\newcommand{\Bqv}[1]{\Big\langle #1 \Big\rangle}

\numberwithin{equation}{section}

\newtheorem{thm}{Theorem}[section]
\newtheorem{lem}[thm]{Lemma}

\newtheorem{rem}[thm]{Remark}
\newtheorem{exm}[thm]{Example}
\newtheorem{ass}[thm]{Assumption}

\newtheorem{Algorithm}[thm]{Algorithm}

\newcommand{\thmref}[1]{Theorem~{\rm \ref{#1}}}

\newcommand{\figref}[1]{Figure~{\rm \ref{#1}}}

\newcommand{\secref}[1]{Section~{\rm \ref{#1}}}

\begin{document}
\title[Numerical approximation for partial observation control of SPDEs]{Numerical approximations for partially observed optimal control of stochastic partial differential equations
}

\author{Feng Bao}
\address{Department of Mathematics, Florida State University, FL 32306}
\email{fbao@fsu.edu}

\author{Yanzhao Cao} 
\address{Department of Mathematics and Statistics, Auburn University, Auburn, AL, 36849}
\email{yzc0009@auburn.edu}

\author{Hongjiang Qian}
\address{Department of Mathematics and Statistics, Auburn University, Auburn, AL 36849}
\email{hjqian.math@gmail.com}

\thanks{The research of F. Bao was supported by the National Science Foundation under grant DMS-214267 and the Department of Energy under grant DE-SC002541; the research of H. Qian and Y. Cao was supported by the Department of Energy under grant DE-SC0022253, DE-SC002564.}

\subjclass[2020]{93E11, 60G35, 65K10, 60H15, 60H10}
\keywords{Partially observed optimal control, numerical approximation, nonlinear filtering, stochastic partial differential equations, stochastic maximum principle, stochastic gradient descent}

\begin{abstract}
In this paper, we study numerical approximations for optimal control of a class of stochastic partial differential equations with partial observations. The system state evolves in a Hilbert space, whereas observations are given in finite- dimensional space $\rr^d$. We begin by establishing stochastic maximum principles (SMP) for such problems, where the system state is driven by a cylindrical Wiener process. The corresponding adjoint equations are characterized by backward stochastic partial differential equations. We then develop numerical algorithms to solve the partially observed optimal control. Our approach combines the stochastic gradient descent method, guided by the SMP, with a particle filtering algorithm to estimate the conditional distributions of the state of the system. Finally, we demonstrate the effectiveness of our proposed algorithm through numerical experiments.
\end{abstract}
\maketitle

\section{Introduction}
For a fixed finite horizon $T>0$ and a bounded domain $\CO=(0,L) \subset \rr$, we consider the following controlled stochastic partial differential equations (SPDEs):
\beq{sde}
dX(t) = [AX(t)+F(X(t),u(t))]dt + G(X(t),u(t)) dW(t), \quad X(0)=X_0 \in L^2(0,L),
\eeq
where $W:[0,T]\times \Omega \to E$ denotes a cylindrical $\CQ$-Wiener process in a separable Hilbert space $E$ on a stochastic basis $(\Omega, \CF, \{\CF_t\}_{t\in [0,T]}, \PP)$ satisfying usual conditions. Here $\CQ$ is a positive self-adjoint nuclear operator on $E$ and $A: D(A) \subset H:=L^2(\CO) \to H$ is a densely defined self-adjoint, negative definite linear operator with domain $D(A)$ and compact inverse. The reaction term $F: H\times U \to H$ and diffusion coefficients $G: H\times U \to \CL(E;H)$ are nonlinear continuous operator depending on a control process $u$ taking values in $\CU$, a subset of another Hilbert space $U$. 

Suppose that the full state of process $X$ is not directly observable, we instead observe a function of $X$ corrupted by noise. It leads to the observation process $Y$ given by 
\beq{Y}
dY(t) = h(X(t))dt + dB(t), \quad Y(0) = 0\in \rr^d
\eeq
where $h: H\to \rr^d$ is a continuous function, and $B(t), t \geq 0$ is a standard Brownian motion in $\rr^d$, independent of $W(t)$. The purpose of the classical filtering problem is to estimate the conditional expectation $\pi_t(\phi)=\EE[\phi(X(t))| \CF_t^Y]$ for a class of functions $\phi$, where $\CF_t^Y$ is the $\sg$-algebra generated by the observation process, that is, $\CF_t^Y:=\sg\{Y(s): s \leq t\}$. For nonlinear filtering of the stochastic partial differential equation, we refer to \cite{Kus70,AFZ97} and references therein.

One would like to determine the control to minimize their cost functional. In \eqref{sde}, the control may depend only on observations and cannot be determined from the state process $X$ only. This control problem is referred to as a partially observed optimal control problem; see Fleming and Pardoux \cite{FP82}, Bensoussan \cite{Ben92}, Pardoux \cite{Par06}, and references therein. It combines classical stochastic optimal control with non-linear filtering problem. In this paper, we focus on systems where the state process follows the stochastic partial differential equations \eqref{sde} driven by a cylindrical Wiener process, and the observation process follows \eqref{Y} in a finite-dimensional space $\rr^d$.

This setting naturally arises in applications where large-scale distributed systems, such as \eqref{sde}, are either physically inaccessible or too high-dimensional for full state monitoring. Instead, partial observations, like \eqref{Y}, provide indirect access to system state. The goal of the controller is to design control strategies based on available but incomplete information $Y$. To further illustrate the motivation of studying partially observed system \eqref{sde} and \eqref{Y}, we take an aquatic system in the ecological problem as an example. Aquatic ecosystems, such as the Great Lakes, are complex environments where water quality and marine life depend on a delicate balance of organic and inorganic agents. The concentration of organic and inorganic such as pollutants and nutrients in the water body can be described by a stochastic  partial differential equation (SPDE) as follows:
\beq{aqu}\left\{\barray\ad 
\frac{\partial C}{\partial t}-D \Dl C + (\nabla C)v = b(C,u)+N_C, \quad t\geq 0, \; \xi\in \CO \\
\ad C|_{\partial \CO} =0, \quad C(0,\xi)= C_0(\xi),\, \xi \in \CO,
\earray\right.\eeq
where $\CO$ is an open, connected, bounded domain representing the aquatic body. The $C$ is the concentration level of $m_1$ different organic and inorganic agents such as pollutants and nutrients. The third term on the left of the first equation \eqref{aqu} represents the transport of $C$ due to water movement, where $v$ is given velocity as a function of space-time. The function $b$ represents the interactions between $m_2$ different control agents $u$ and the $m_1$ different pollutants and nutrients $C$. The $N_C$ is the distributed noise representing the additive effect of randomness, such as land run-offs from surrounding farmlands, acid rain, accidental oil spills, and summer cottages, etc. The aquatic system supports diverse species, including microorganisms and fish populations, where the stock of fish is regulated by the Department of Fisheries. The biomass per unit volume of $m_3$ different species of population $y$ is governed
by 
\beq{ob-h}
dy= h(C,y)dt+ N_y,\quad y(0)=y_0, t\geq 0
\eeq
where $h(C,y)$ represents the growth of the species and $N_y$ accounts for random environmental factors. For example, the function $h$ can take a standard logistic growth function. Regulatory agencies, such as the Department of Fisheries and Environments, are interested in seeking optimal strategies to maintain water quality and promote marine life, e.g. by applying antipollutants, biological agents predating unwanted microorganisms, physical removal of solid water, algae, etc. They would like to minimize their cost functional to achieve goals. For the formulation of \eqref{aqu} as an SPDE \eqref{sde} and its corresponding partially observed optimal control problem, we refer to \cite[pp. 1593-1597]{Ahm96}. Another example from electromagenetic interference control problem can also be found in \cite{Ahm96}. 

The primary challenge of partially observed control problem is that the control process must be adapted to the observation filtration, while  observations depends on controlled process, leading to circular dependency. To overcome this difficulty, we use the measure transformation approach to introduce
\beq{M}
M(t):=\exp\Big\{\int_0^t h(X(s))dY(s)-\half \int_0^t |h(X(s))|^2 ds\Big\},
\eeq
where $h$ is bounded continuous mapping from $H$ to $\rr^d$. Then $M(t)$ is the unique solution to stochastic ordinary differential equations (SODEs):
\beq{Mt}
dM(t) = M(t) h(X(t))dY(t),\quad M(0)= 1.
\eeq
Define a new measure $\QQ$ by 
\bea\ad 
\frac{d\QQ}{d\PP}\Big|_{\CF_T}:= M(T)
\eea
It follows from Girsanov's theorem \cite{KS91} that $W\cd$ and $B\cd$ are cylindrical Wiener process and Brownian motions in $\rr^d$, resepectively, in the new probability space $(\Omega,\CF,\{\CF_t\}_{t\in [0,T]}, \QQ)$. We would like to remark that we have considered $Y$ a prior as a Wiener process in $(\Omega,\CF,\{\CF_t\}, \PP)$, rather than $B$ from modeling perspective; see \cite{WWX24}. The cost functional is given by 
\beq{Ju}\barray
J(u)\ad =\EE^\QQ \Big[\int_0^T L(X(t), u(t))dt+\Phi(X(T))\Big] \\
\aad = \EE^\PP \Big[\int_0^T M(t) L(X(t),u(t))dt+ M(T)\Phi(X(T))\Big],
\earray\eeq
for running cost functional $L: H\times U \to \rr$ and terminal cost $\Phi: H \to \rr$, where $\EE^\QQ$ and $\EE^\PP$ represents expectation on $(\Omega, \CF, \{\CF_{t}\}_{t\in [0,T]},\QQ)$ and $(\Omega, \CF, \{\CF_t\}_{t\in [0,T]}, \PP)$, respectively. Throughout the paper, we take $\EE$ as $\EE^\PP$ if there is no specification. The control problem is to minimize $J(u)$ defined \eqref{Ju} subject to \eqref{sde} and \eqref{Mt}. That is, 
\medskip

\noindent \textbf{Problem}: find $u^\star \in \CU_{ad}[0,T]$ such that
\beq{contr-prob}
J(u^\star)= \inf_{u \in \CU_{ad}[0,T]} J(u).
\eeq
where $\CU_{\text{ad}}[0,T]$ is the set of admissible controls defined as 
\beq{control}
\CU_{ad}[0,T]=\{u: [0,T]\times \Omega \to U | u \text{ is } \CF^Y\text{-progressively measurable}\}.
\eeq

The primary objective of this paper is to formulate and analyze a fully implementable numerical algorithm to solve the partially observed optimal control problem in \eqref{sde}, \eqref{Mt}, \eqref{Ju}, and \eqref{contr-prob}. Our approach leverages the stochastic maximum principle (SMP) to develop a stochastic gradient descent (SGD) algorithm to compute the optimal control. This is combined with a particle filtering algorithm to compute the conditional distribution of the state process given the observations. To the best of our knowledge, there are no existing results on numerical approximations for partially observed optimal control problems where the state process is governed by a SPDE \eqref{sde}, and observations evolve in  $\rr^d$. We aim to fill this gap in this paper.

The optimal control of partially observed diffusion processes is a well-known challenging problem in literature. A classical approach to address this problem is to use the ``separation principle'', first introduced by Wonham \cite{Won68}. This principle allows the problem to be solved in two steps: (i) estimating the system state using the noisy observation process and (ii) applying a memoryless function of this estimate as a control input. Another widely used technique is reformulating the partially observed control problem into a completely observed stochastic control problem, where the state dynamics are governed by the Duncan-Mortensen-Zakai (DMZ) equation, a linear SPDE. This reformulation results in an infinite-dimensional optimal control problem, often called the ``separated problem'', dating back to the work of Fleming and Pardoux \cite{FP82}. For further developments on partially observed control problems in finite-dimensional spaces, see \cite{GS00,BV75,NN90,Zha90,Zho92} and references therein.

In recent years, stochastic optimal control of SPDEs has gained increasing attention. For comprehensive study of stochastic optimal control for SPDEs without partial observations,  we refer to \cite{DM13,FGS17,FZ20,FO16,FHT18,LZ14,LZ18} and references therein. However, the literature on partially observed optimal control problems where state evolves in infinite-dimensional spaces remains relatively sparse. A notable contribution is Ahmed \cite{Ahm96}, who established the existence of optimal relaxed controls for SPDEs with finite-dimensional observations, assuming the diffusion coefficient takes the form $G(X(t),u(t))=\sqrt{\CQ}$, where $\CQ$ is a symmetric positive operator in $H$, and the control appears in feedback form, i.e.,  $F(X,u)=F(X)+B(X,u(t,Y))$ for some continuous bounded maps $B: H \times U \to H$. Later, Ahmed \cite{Ahm19} extended these results to a general class of nonlinear partially observed stochastic evolution equations in Banach spaces, proving the existence of optimal feedback controls (see also \cite{Ahm14} for related results in infinite-dimensional spaces). 

While a rigorous mathematical framework for partially observed control problems exits under suitable conditions in both finite and infinite-dimensional space, most stochastic differential equations (SDEs) and SPDEs lack closed-form solutions. Consequently, numerical approximations play a crucial role in solving these problems.

In this paper, we assume the existence of an optimal control for the partially observed control problem and focus on developing numerical algorithms to approximate it. In the classical stochastic optimal control problem, two major frameworks exist for numerical computation: (i) method based on stochastic maximum principle (SMP), and (ii) method based on the dynamic programming principle (DPP). A well-established DPP-based methodology is the Markov chain approximation method, introduced by Kushner \cite{Kus77,KD01}. This technique has been extensively applied to find the optimal control for stochastic differential equations, including those with regime-switching and random jumps; see \cite{YWQN21,QWY22} for completely observed control problem and recent work \cite{LTT24} for partially observed ones.  However, this approach suffers from the curse of dimensionality, as it requires solving high-dimensional Hamilton-Jacobi-Bellman (HJB) partial differential equations. 

Recently, SMP-based methods have gained popularity as an efficient alternative; see \cite{GLTZZ17}. These methods compute  gradients using the stochastic maximum principle and use stochastic gradient descent (SGD) algorithm to iteratively update the control. In the context of partial observations, a nonlinear filtering  algorithm is incorporated to compute the conditional distribution of the state process given observations. Notably, Archibald et al. \cite{ABYZ20} developed SGD algorithms combined with a particle filtering algorithm to compute the optimal control. A related approach was proposed by Liang et al. \cite{LHBAZ24} who used backward SDE filter in \cite{BM17}. More recently, Wan et al. \cite{WWX24} studied a setting where the system state and observation process are driven by correlated noise and introduced a branching particle filter algorithm in combination with SGD to compute the partially observed optimal control.

Our main contributions are twofold. First, we established the stochastic maximum principle for partially observed optimal control problems using a measure transformation technique. The system state evolves in an infinite-dimensional Hilbert space, while the observation process is a finite-dimensional process in $\rr^d$. Second, we used the SMP to develop a stochastic gradient descent (SGD) algorithm to compute the optimal control. It consists of the finite-element approximation for forward-backward SPDEs and the particle filtering algorithm to estimate the conditional distribution of the state given observations.

The structure of the paper is as follows. In \secref{sec:pre}, we introduce the necessary assumptions and derive the stochastic maximum principle for the partially observed optimal control problem \eqref{sde}, \eqref{Mt}, and \eqref{Ju}. In this framework, the first-order adjoint equations are characterized by both backward stochastic partial differential equations (BSPDEs) and backward stochastic differential equations (BSDEs). In \secref{sec:fe}, we use finite-element method to discretize the forward-backward stochastic partial differential equations (FB-SPDEs) in space and apply the implicit Euler method for time discretization to obtain the numerical approximation scheme for FB-SPDEs. In \secref{sec:alg}, we use the stochastic maximum principle established in \secref{sec:pre} to construct a SGD algorithm combined with the particle filtering algorithm to approximate the optimal partially observed control. Finally, \secref{sec:num} provides numerical examples to illustrate the effectiveness of our approach.

\section{Preliminaries and main results} \label{sec:pre}
\subsection{Notation and assumptions}
Given two real separable Hilbert spaces $E,E'$, $\CL(E,E')$ denotes the space of bounded linear operators from $E$ to $E'$, endowed with the usual operator norm. The $\CL_1(E,E')$ denotes the subspace of trace class operators and $\CL_2(E;E')$ is the space of Hilbert-Schmidt operators, endowed with the Hilbert-Schmidt norm. For $E=E'$, we write $\CL(E), \CL_1(E)$, and $\CL_2(E)$ instead of $\CL(E,E'), \CL_1(E,E')$, and $\CL_2(E,E')$. We denote by $H:=L^2(\CO)$ and use $|\cdot|_H, \qv{\cdot,\cdot}_H$ to denote the norm and inner product in $H$, respectively. The Euclidean norm in $\rr^d$ is denoted by $|\cdot|$ and the corresponding inner product is $\qv{\cdot,\cdot}_{\rr^d}$. 
Here and below, we use the symbol $|\cdot|,\qv{\cdot,\cdot}$ to denote a norm and the inner product when the corresponding space is clear from the context, otherwise we use a subscript. For a nonlinear map $F: H\times U \to H$, we denote by $\nabla_x F,\nabla_u F$ the corresponding G\^{a}teaux derivative with respect to state variable and control variable, respectively. Their adjoint are denoted as $\nabla_x^* F,\nabla_u^* F$. Throughout the paper, we use $K$ as a generic constant which may vary from place to place. 

We will use the following class of processes throughout the paper.
\begin{itemize}
\item $L_\CP^2(\Omega; L^2([0,T];H))$: the space of all predictable $H$-valued process $X: \Omega \times [0,T] \to H$ satisfying $\EE\int_0^T |X(t)|_H^2 dt <\infty$.

\item $L_\CP^2(\Omega; C([0,T];H))$: the space of all predictable $H$-valued continuous process $X: \Omega \times [0,T] \to H$ satisfying $\EE[\sup_{t\in [0,T]}|X(t)|_H^2] <\infty$. 

\item $L_\CF^2([0,T];H)$: the space of stochastic process $X$ with values in $H$, adapted to filtration $\CF_t$ such that $\EE \int_0^T |X(t)|_H^2 dt <\infty$.
\end{itemize}

Let $\{e_n\}_{n\geq 0}$ be an orthonormal basis of $H$ consisting of eigenfunctions of $A$ with corresponding eigenvalues $\{-\la_n\}_{n\geq 0}$. For any $r\in \rr$ we can define the fractional operator $(-A)^{r/2}: D((-A)^{r/2}) \to H$ by 
\bea\ad
(-A)^{r/2}x:= \sum_{n=1}^\infty \la_n^{r/2} x_n e_n 
\eea
for all
\bea\ad 
x\in D((-A)^{r/2}) :=\Big\{x= \sum_{n=1}^\infty x_n e_n\in H: \|x\|_r^2 :=\sum_{n=1}^\infty \la_n^r x_n^2 <\infty \Big\}.
\eea
Let $H^r:= D((-A)^{r/2})$. Then $\|x\|_r  =\|(-A)^{r/2}x\|_H$ defines a norm on $H^r$.

\begin{rem}{\rm
If $A=\Dl$, where $\Dl$ denotes the Laplace operator with Dirichlet boundary conditions, it is well-known that $H^1=H_0^1(\CO)$ and $H^2=H^2(\CO) \cap H_0^1(\CO)$; see \cite{Kru14}. In the case of Neumann boundary conditions, we can consider the operator $(A-\al I)$ for some constant $\al>0$ and $F(x,u)+\al x$ in the nonlinearity instead.
} 
\end{rem}
Consider a separable Hilbert space $V$, continuously embedded in $H$. Its dual is denoted by $V'$. We will consider the solution of \eqref{sde} on the Gelfand triple
\bea
V \hookrightarrow H \hookrightarrow V'
\eea
where $V=H^1$ equipped with norm $\|\cdot\|_{H^1}^2=|\cdot|_H^2 + \|\cdot\|_1^2$. 
We assume the following assumptions on the coefficients of partially observed control system \eqref{sde} and \eqref{Y}.

\begin{ass} \label{ass} For any $x,y\in H$, we assume 
\begin{itemize}
\item[(H1)] $A \in \CL(V;V')$ and there exist constant $K_1>0$ and $\la \geq 0$ such that 
\bea
\qv{-A v, v}+ \la |v|^2 \geq K_1 \|v\|^2, \quad \forall\, v \in V.
\eea
\item[(H2)] For $u \in U$, $F(x,u): H\times U \to H$ and $G(x,u): H\times U \to \CL(E; H)$ are G\^{a}teaux-differentiable, with continuous bounded derivatives $\nabla_x F, \nabla_u F, \nabla_x G, \nabla_u G$. It implies the following Lipschitz continuity of $F,G$:
\bea
|F(x,u)-F(y,u)|_H \leq  K |x-y|_H \\
\|G(x,u)-G(y,u)\|_{\CL(E;H)} \leq K |x-y|_H
\eea
\item[(H3)] The running cost functional $L(x,u): H\times U \to \rr$ and terminal cost $\Phi: H \to \rr$ are G\^{a}teaux differentiable. Moreover $\nabla_x L , \nabla_u L$, and $\nabla_x \Phi$ are continuous and satisfy the following growth condition:
\bea
|\nabla_x L| , |\nabla_u L|, |\nabla_x \Phi| \leq K(1+|x| + |u|).
\eea
\item[(H4)] The function $h: H \to \rr^d $ is bounded and continuous functions.
\end{itemize}
\end{ass}
Under above assumptions, for any admissible control $u\in \CU_{\text{ad}}$, equation \eqref{sde} has a unique probabilistic strong solution in the variational setting on the Gelfand tripe. Moreover, $X\in L_\CF^2 (0,T;V) \cap L^2(\Omega, C([0,T];H))$. The proof of existence and uniqueness for solution of \eqref{sde} is standard and can be found in Bensoussan \cite[Theorem 1]{Ben83} and \cite{LR15}.

\subsection{A reference example} Let us consider the following controlled stochastic heat equation in the interval $(0,1)$, perturbed by multiplicative noise:
\beq{srde}
\left\{\barray
\disp \frac{\partial X_t}{\partial t} (\xi)=\frac{\partial^2}{\partial \xi^2}X_t(\xi)+f(\xi,X_t(\xi),u_t(\xi))+g(\xi,X_t(\xi), u_t(\xi))\frac{\partial w^\CQ}{\partial t}(t,\xi)\\
X_1(0)=X_t(1)=0,\; t\in [0,T]\\
X_0(\xi)=x_0(\xi),\; \xi\in [0,1],
\earray\right.
\eeq
where $f,g: [0,1]\times \rr\times \rr\to \rr$ are given Borel measurable functions and $\partial w^\CQ/\partial t (t,\xi)$ is a $\CQ$-Wiener process. We assume the mapping $f(\xi,\cdot,\cdot), g(\xi,\cdot,\cdot)$ are of class $C^1$, Lipschitz continuous, uniformly with respect to $\xi\in [0,1]$, and that $f(\cdot,0,0)$ and $g(\xi,0,0)$ are bounded. The set of admissible control actions $\CU$ is a convex subset of $U:=L^2([0,1])$ and we assume that $\CU \subset L^\infty([0,1])$. A control $u$ is a progressively process with values in $\CU$. The cost functional is:
\beq{cost}
J(u)= \EE \int_0^T \int_0^1 \ell(\xi,X_t(\xi), u_t(\xi))d\xi dt + \EE \int_0^1 \phi(\xi, X_T(\xi))d\xi,
\eeq
where $\ell: [0,1]\times \rr \times \rr \to \rr$ and $\phi:[0,1]\times \rr \to \rr$ are given bounded, Borel measurable functions satisfying suitable conditions. The classical control problem for stochastic partial differential equation is to find $u^\star \in \CU[0,T]$ to minimize the cost function $J(u)$, that is,
\bea\ad 
J(u^\star)=\inf_{u\in \CU[0,T]} J(u).
\eea
Let $H:=L^2(0,1)$, the SPDE \eqref{srde} can be reformulated in Hilbertian framework resulting in \eqref{sde}, where $F,G,L,\Phi$ can be defined as the following Nemytskii operators: for $x,y\in H, u\in L^\infty(0,1)$,
\bea
F(x,u)(\xi)\ad = f(\xi, x(\xi), u(\xi)), \quad [G(x,u)y](\xi)= g(\xi, x(\xi), u(\xi)) y(\xi)\\
L(x,u)\ad =\int_0^1 \ell(\xi, x(\xi), u(\xi))d\xi, \quad \Phi(x) =\int_0^1 \phi(\xi, x(\xi))d\xi.
\eea
Now we suppose the state of controlled stochastic heat equation $X_t$ is not completely observable, instead, one observes a function of $X_t$ subject to some noise giving rise to observation process \eqref{Y}. The control problem becomes to optimize \eqref{cost} based on information of $Y$.

\subsection{Stochastic maximum principle}\label{sec:SMP}
In this subsection, we devote to establishing stochastic maximum principles (SMP) for the partially observed control system \eqref{sde}, \eqref{Mt}, and \eqref{Ju}. When there is no observation process, i.e., in the fully observed case, Bensoussan \cite{Ben83} derived an SMP for stochastic evolution equations driven by a cylindrical $\CQ$-Wiener process using a variational approach. Later, Fuhrman et al. \cite{FHT18} extended Pontryagin's maximum principle to the case where the stochastic partial differential equation \eqref{sde} is driven by space-time white noise (i.e., $\CQ=I$) under a mild solution framework. They also established the well-posedness of the associated adjoint backward stochastic partial differential equations (BSPDEs). Notably, in both \cite{Ben83} and \cite{FHT18}, the BSPDE contains a series in the drift coefficients arising from the cylindrical Wiener process in the state process. For further details on SMP for SPDEs, we refer to \cite{FHT13, FHT18, DM13} and the references therein.

For control problems under partial observations, Bensoussan and Viot \cite{BV75} established necessary optimality conditions for linear stochastic distributed parameter systems using a variational approach. They considered a feedback control structure where the function $h$ in \eqref{Y} takes the form $h(X(t))=D(t)X(t)$, with $D\in L^\infty([0,T]; \CL(H;E))$ for a Hilbert space $E$. Later, Ahmed \cite{Ahm96} derived an SMP for the stochastic partial differential equation \eqref{sde} with $G(X(t),u(t))=\sqrt{\CQ}$ subject to the observation process \eqref{Y}, using Zakai equation. However, handling the infinite-dimensional filtering problem in this approach remains significantly challenging and it is difficult to apply corresponding stochastic maximum principle.

We will establish a general maximum principle for SPDEs with partial observations using the Girsanov transformation, under the assumption that the control set is convex. For nonconvex control domains, the stochastic maximum principle becomes significantly more intricate, and we leave it to our subsequent work. The measure transformation technique that we employ is well known in the finite-dimensional setting. We refer to \cite{LT95, Tan98} for details.

Let $u^\star \in \CU_{\text{ad}}$ be the optimal control and $X^\star$ the corresponding optimal trajectory of \eqref{sde}. The goal of this section is to find the necessary conditions that are satisfied by $u^\star$. Let us define the Hamiltonian as follows: 
\beq{Ham}\barray
H(t,X,u,Q,q_1,C,c_2) \ad :=\qv{A X(t) + F(X(t), u(t)), Q(t)}_H + \text{Tr}[q_1^*(t) G(X(t),u(t)) \CQ] \\
\aad\quad + L(X(t),u(t))+ \qv{h(X(t)),c_2(t)}_{\rr^d}
\earray\eeq
where 
\bea\ad 
\text{Tr}[q_1^*(t) G(X(t),u(t))]:=\sum_{i=1}^\infty \qv{G(X(t),u(t)) \CQ e_i, q_1(t)e_i}_{H}
\eea
and $\{e_i\}$ is an orthonormal basis of Hilbert space $H$. We introduce the following backward stochastic partial and ordinary differential equations: 
\beq{BSPDE}\left\{\barray
dQ(t) \ad = -\Big[A^* Q(t)+ \nabla_x^* F(X(t),u(t)) Q(t)+\sum_{i=1}^\infty \nabla_x^*\big[G(X(t), u(t)) \CQ e_i \big] q_1(t)e_i  \\
\aad \qquad +\nabla_x L(X(t),u(t))+ \sum_{k=1}^d \nabla_x^* h^k(X(t))c_2^k(t) \Big] dt + q_1(t) dW(t) + \sum_{k=1}^d q_2^k(t) dB^k(t), \\
dC(t) \ad = - L(X(t),u(t)) dt + c_1(t) dW(t) + \sum_{k=1}^d  c_2^k(t)dB^k(t) \\
Q(t) \ad = \nabla_x \Phi(X(T))\\
 C(t)\ad = \Phi(X(T))
\earray\right.\eeq
where $q_1\cd,q_2^k\cd$ is the martingale representation of $Q\cd$ with respect to $W$ and $B^k$, respectively, and $c_1\cd,c_2^k\cd$ are the martingale representation of $C\cd$ with respect to $W$ and $B^k$, respectively. We have two adjoint backward differential equations in \eqref{BSPDE} because our system state has two components $(X\cd, M\cd$). 

From now on, we adopt the convention of finite summation over repeated indices, so that we will drop the symbol $\sum_{k=1}^d$ in \eqref{BSPDE} and other equations.

Under our assumptions (H1)-(H4), it can be shown that there exist unique $\CF_t$-adapted solutions: $Q\in L^2(\Omega\times [0,T];V)\cap L^2(\Omega; C([0,T];H)), q_1 \in  L^2(\Omega\times [0,T]; \CL_2(H)), q_2^k \in L^2(\Omega\times [0,T]; H)$ and $C \in C([0,T]; \rr)$, $ c_1\in L^2(\Omega\times [0,T];  \CL_2(H;\rr))$, $c_2^k \in L^2(\Omega\times [0,T]; \rr)$. The existence and uniqueness of backward SPDE and SDE in \eqref{BSPDE} can be found in \cite{Ben83, FHT18, LT95}, thus details are omitted. We have the following main result for the stochastic maximum principle of our partially observed control problem.

\begin{thm}\label{thm:SMP}
Let assumptions (H1)-(H4) hold and assume that $u^\star$ is an optimal control and $X^\star$ is the corresponding optimal state. Then for any $v\in \CU_{\text{ad}}$, it is necessary to satisfy
\bea
\ad \bqv{\EE^\QQ [\nabla_u H(t,X^\star(t),u^\star(t),Q^\star(t),q_1^\star(t),C^\star(t), c_2^\star(t))|\CF_t^Y], v-u^\star(t)}
\\
\ad = 
\Bqv{\EE^\QQ \Big[\nabla_u^* F(X^\star(t), u^\star(t)) Q^\star(t) + \nabla_u L(X^\star(t),u^\star(t))\\
\aad \qquad +\sum_{i=1}^\infty \nabla_u^*[G(X^\star(t),u^\star(t)) \CQ e_i]q_1^\star (t)  e_i \Big|\CF_t^Y\Big], v-u^\star(t)} \geq 0,
\eea
where $\{Q^\star,q_1^\star,q_2^\star \}$ and $\{C^\star, c_1^\star, c_2^\star\}$ are solutions of \eqref{BSPDE} with $X\cd,u\cd$ replaced by $X^\star, u^\star$.
Define $\nabla_u H(u^\star):= \nabla_u F(X^\star, u^\star) Q^\star(t) + \nabla_u L(X^\star,u^\star)+\sum_{i=1}^\infty \nabla_u^*[G(X^\star, u^\star)\CQ e_i]q_1^\star e_i$ as the gradient of the Hamiltonian \eqref{Ham} with respect to the optimal control $u^\star$. Then the G\^{a}teaux derivative of the cost function $J$ at $u^\star$ satisfies
\bea
\nabla J(u^\star)= \EE^\QQ[\nabla_u H(u^\star)|\CF_t^Y].
\eea
\end{thm}

To prove \thmref{thm:SMP}, we follow the work of Li and Tang \cite{LT95,Tan98} on SMP for thepartially observed optimal control in finite-dimensional space and Bensoussan \cite{Ben83} for the  completely observed optimal control for SPDEs in infinite-dimensional space. We set 
\beq{z}
z:=\begin{pmatrix}
X \\ M
\end{pmatrix},\quad 
z_1:=\begin{pmatrix}
X_1 \\ M_1
\end{pmatrix},
\eeq
and let 
\bea
L(z,u)=ML(X,u), \quad \Phi(z):= M \Phi(X).
\eea
The cost functional \eqref{Ju} can be represented by 
\bea\ad 
J(u)=\EE\Big[\int_0^T L(z(t),u(t))+ \Phi(z(T))\Big].
\eea

The following lemma gives the G\^{a}teaux derivative of $J$ with respect to the control variable.

\begin{lem}\label{lem:grad-J}
The functional $J\cd$ is G\^{a}teaux differentiable and the following formula holds:
\bea\ad 
\frac{d}{d\theta} J(u\cd+\theta v\cd)\Big|_{\theta=0} \\
\aad=\EE^\QQ \Big\{\int_0^T \qv{\nabla_x L(X(t),u(t)),X_1(t)} +\qv{\nabla_u L(X(t),u(t)), v(t)} dt \\
\aad \qquad\qquad + \qv{\nabla_x \Phi(X(T)), X_1(T)} \Big\} \\
\aad\quad + \EE^\QQ \Big\{\int_0^T M^{-1}(t)M_1(t) L(X(t),u(t))dt + M^{-1}(T)M_1(T) \Phi(X(T))\Big\}. 
\eea
where $X_1$ and $M_1$ are solutions of the following equations:
\beq{X1}\barray
dX_1(t) \ad = [A X_1(t)
dt+ \nabla_x F(X(t),u(t)) X_1(t)+ \nabla_u F(X(t),u(t))v(t)] dt\\
\aad\quad + [\nabla_x G(X(t), u(t)) X_1(t)+ \nabla_u G(X(t), u(t)) v(t)]dW(t), \quad X_1(0)=0,
\earray\eeq
and 
\beq{M1-eq}
d M_1 (t)= [ M_1 (t) h(X(t))+ M(t) \nabla_x h(X(t)) X_1(t)] dY(t), \quad M_1(0)=0.
\eeq
The equations \eqref{X1} and \eqref{M1-eq} are first-order variational equation for state \eqref{sde} and \eqref{Mt}.
\end{lem}
\begin{proof}
For any $\theta\in (0,1)$ and $v\in \CU_{\text{ad}}$. The convexity of $\CU_{\text{ad}}$ implies that $u+\theta v\in \CU_{\text{ad}}$. Let $X_\theta\cd$ and $M_\theta\cd$ be the trajectory of \eqref{sde} and \eqref{Mt} corresponding to control $u\cd+\theta v\cd$. Define $z_\theta, z_1$ as in \eqref{z} using $X_\theta$ and $M_\theta$. Denote by $\wdt X_\theta(t) = (X_\theta(t)-X(t))/\theta - X_1(t)$, where $X_1\cd$ is the solution of \eqref{X1} and denote by $\wdt M_\theta= (M_\theta(t)-M(t))/\theta - M_1(t)$, where $M_1 \cd$ satisfies \eqref{M1-eq}. Then one can show that as $\theta \to 0$,
\beq{wdt-XM}
\sup_{t\in [0,T]}\EE|\wdt X_\theta(t)|_H^2 \to 0 \text{ and } \sup_{t\in [0,T]} \EE |\wdt M_\theta(t)|^2 \to 0.
\eeq
Moreover, we have
\beq{dJ}\barray\ad 
\frac{J(u\cd+\theta v\cd)-J(u\cd)}{\theta}\\
\aad = \frac{1}{\theta}\Big\{\EE \int_0^T L(z_\theta(t),u(t)+\theta v(t))-L(z(t),u(t))dt +\EE[ \Phi(z_\theta(T)) -\Phi(z(T))] \Big\} \\
\aad =  \EE\int_0^T \int_0^1 \qv{ \nabla_z L(z(t)+\la(z_\theta(t)-z(t)), u(t) + \la \theta v(t)), z_1(t)+\wdt z_\theta(t)}  \\
\aad\qquad\qquad\quad + \qv{\nabla_u L(z(t)+\la(z_\theta(t)-z(t)), u(t)+\la \theta v(t)), v(t)} \, d\la dt \\
\aad\quad + \EE \qv{\nabla_z \Phi(z_\theta(T)+\la(z_\theta(T)-z(T))), z_1(T)+\wdt z_\theta(T)}.
\earray\eeq
For details of \eqref{wdt-XM} and \eqref{dJ}, we refer to \cite[Lemma 2.1]{Ben83} and \cite[Lemma 3.2, 3.3]{LT95}.
Letting $\theta \to 0$, we can have
\bea\ad 
\frac{d}{d\theta} J(u\cd + \theta v\cd)\Big|_{\theta=0}\\
\aad =\EE \int_0^T \qv{\nabla_z L(z(t),u(t)), z_1(t)} + \qv{\nabla_u L (z(t), u(t)), v(t)} dt +\qv{\EE \nabla_z \Phi(z(T)), z_1(T)}  \\
\aad = \EE \int_0^T \big\{\qv{M(t) \nabla_x L(X(t),u(t)), X_1(t)}  + M_1(t) L(X(t),u(t)) \\
\aad\qquad\qquad +\qv{M(t) \nabla_u L(X(t),u(t)), v(t)} \big\} dt \\
\aad \quad + \EE \qv{M(T)\nabla_x \Phi(X(T)),X_1(T)} + \EE M_1(T)\Phi(X(T))
\eea
Consequently,
\beq{grad-J}\barray\ad 
\frac{d}{d\theta} J(u\cd + \theta v\cd)\Big|_{\theta=0} \\
\aad = \EE^\QQ\Big\{ \int_0^T \qv{\nabla_x L(X(t), u(t)), X_1(t)} + \qv{\nabla_u L(X(t), u(t)), v(t)}  dt \\
\aad \qquad \qquad  +\EE^\QQ \qv{\nabla_x \Phi(X(T)), X_1(T)}  \Big\} \\
\aad\quad + \EE^\QQ \Big\{\int_0^T M^{-1}(t) M_1(t) L(X(t), u(t))dt +  M^{-1}(T) M_1(T)\Phi(X(T))\Big\}
\earray\eeq
The proof is completed.
\end{proof}

\begin{proof}[Proof of \thmref{thm:SMP}]
By It\^{o} formula, we can verify that the solution of \eqref{M1-eq} have the following explicit expressions (see \cite{BEK89}):
\bea\ad 
M_1(t)= M(t) \int_0^t \nabla_x h(X(t)) X_1(t) dB(t). 
\eea
Let $\Ga(t)=M^{-1}(t)M_1(t)$, then we have 
\beq{Ga-t}
\Ga(t)=\int_0^t \nabla_x h(X(t)) X_1(t)dB(t).
\eeq
In particular, we have $\Ga(0)=0, \Ga(T)=M^{-1}(T)M_1(T)$. Applying It\^{o} formula to $C(t)\Ga(t)$, we obtain
\bea\ad 
\EE^\QQ C(T)\Ga(T)=\EE^\QQ[C(0)\Ga(0)]+\EE^\QQ \int_0^T -L(X(t),u(t))\Ga(t)+ \qv{\nabla_x^* h^k(X(t))c_2^k(t),X_1(t)} dt,
\eea
where $h\cd=(h^1\cd, \dots, h^k\cd)$. It implies that
\beq{M1}\barray\ad 
\EE^\QQ \Big[M^{-1}(T)M_1(T)\Phi(X(T))+ \int_0^T M^{-1}(t)M_1(t)L(X(t),u(t)) \Big]\\
\aad= \EE^\QQ \int_0^T \qv{\nabla_x^* h^k(X(t))c_2^k(t),X_1(t)} dt
\earray\eeq
Applying It\^{o} formula to $\qv{X_1(t),Q(t)} $ and taking integration and  expectation, we obtain
\beq{X1-P}\barray\ad
\EE^\QQ [\nabla_x \Phi(X(T)) X_1(T)] \\
\ad =\EE^\QQ \qv{Q(T),X_1(T)} \\
\ad = \EE^\QQ \int_0^T - \qv{ \nabla_x L(X(t), u(t)), X_1(t)}_H - \qv{\nabla_x^* h^k(X(t)) c_2^k(t), X_1(t)} dt \\
\aad + \EE^\QQ \int_0^T \qv{\nabla_u F(X(t),u(t)) v(t),Q(t)} + \text{Tr}[q_1^*(t) \nabla_u G(X(t),u(t)) \CQ v(t)] dt
\earray\eeq
Combining \eqref{grad-J}, \eqref{M1}, and \eqref{X1-P}, we get
\beq{re-grad-J}\barray\ad 
\frac{d}{d\theta} J(u\cd +\theta v\cd)\Big|_{\theta=0} \\
\aad = \EE^\QQ \int_0^T \Big\{ \qv{\nabla_u L(X(t),u(t)),v(t)}_H +\qv{\nabla_u F(X(t),u(t))v(t),Q(t)} \\
\aad \qquad \qquad +\text{Tr}[\nabla_u G(X(t),u(t)) v(t) q_1(t)] \Big\} dt\\
\aad =\EE^\QQ \int_0^T \qv{\nabla_u^* F(X(t),u(t))Q(t)+\nabla_u L(X(t),u(t)) \\
\aad\qquad \qquad +\sum_{i=1}^\infty \nabla_u^* [ G(X(t),u(t))e_i]q_1(t)e_i, v(t)} \, dt
\earray\eeq
Let $u^\star$ be the optimal control which is in the interior of $U$, then for any $v\in \CU_{\text{ad}}$, \eqref{re-grad-J} implies
\bea\ad 
\EE^\QQ\int_0^T \qv{\nabla_u^* F(X^\star(t),u^\star(t))Q^\star(t)+ \nabla_u L(X^\star(t),u^\star(t)) \\
\aad \qquad\qquad +\sum_{i=1}^\infty \nabla_u^*[G(X^\star(t),u^\star(t))e_i] q_1^\star(t)e_i, v-u^\star} dt \geq 0,
\eea
where $(Q^\star, q_1^\star, q_2^{k,\star})$ are solutions of \eqref{BSPDE} with $(X(t),u(t))$ replaced by $(X^\star(t), u^\star(t))$. Since the control variable must be adapted to observation filtration $\CF_t^Y$, we have for a.s. $t \in [0,T]$,
\beq{nabla-J}\barray
\nabla J(u^\star)\ad :=\EE\Big[\nabla_u^* F(X^\star(t), u^\star(t)) Q(t)+\nabla_u L(X^\star_t,u^\star_t) \\ 
\aad\qquad \quad + \sum_{i=1}^\infty \nabla_u^* [G(X^\star, u^\star(t))e_i]q_1^\star(t) e_i \Big| \CF_t^Y \Big] = 0.
\earray\eeq
We finish the proof of stochastic maximum principle in \thmref{thm:SMP}.
\end{proof}

In what follows, we apply the stochastic maximum principle established in \thmref{thm:SMP} to develop numerical algorithms to approximate the partially observed optimal control. In \secref{sec:fe}, we use the finite element method for spatial discretization and the implicit Euler method for temporal discretization to obtain numerical schemes for the forward-backward SPDE system \eqref{sde} and \eqref{BSPDE}. Then, in \secref{sec:alg}, we construct a stochastic gradient descent algorithm, combined with a particle filtering approach, to approximate the optimal control.

\section{Spatial-time discretization of forward-backward SPDEs}\label{sec:fe}
In this section, we focus on discretizing forward-backward SPDEs \eqref{sde} and \eqref{BSPDE} in both space and time to obtain their numerical solutions. From \secref{sec:SMP}, it follows that the gradient of the cost functional $J$ with respect to the control process on the time interval $t\in [0,T]$ is given by \eqref{grad-J}, where the corresponding adjoint backward stochastic partial differential equations are given by \eqref{BSPDE}. In practical applications, controlling the diffusion term is often challenging. Therefore, we consider a setting where the diffusion coefficients of the state of the system are independent of the control variable, that is, $G(X(t), u(t))= G(X(t))$. Additionally, we assume that the control variable appears in additive form, meaning that the drift function takes the structure $F(X(t),u(t)):=F(X(t))+u(t)$, where $F: H \to H$ is a non-linear operator. For simplicity, we take the second-order differential operator $A=\Dl$ for the rest of the paper.

Under this setting, the gradient of the cost functional in \eqref{nabla-J} simplifies to
\beq{nabla-J1}
\nabla J(u^\star)=\EE \Big[\nabla_u^* F(X^\star,u^\star) Q^\star(t)+\nabla_u L(X^\star, u^\star)\Big| \CF_t^Y\Big].
\eeq

Moreover, in \secref{sec:SMP}, we observed that the backward stochastic partial differential equations satisfied by $Q(t)$ in \eqref{BSPDE} contains a series in the drift term. From a computational perspective, we consider system state \eqref{sde} driven either by an additive cylindrical Wiener process or by finitely many Brownian motions with multiplicative noise. In the case of additive cylindrical Wiener process, the series term in \eqref{BSPDE} vanishes, simplifying computations. On the other hand,  when the noise is multiplicative and driven by finitely many Brownian motions, the series reduces to a finite summation, making it computational tractable.

For the remainder of the paper, we formulate a numerical algorithm under the assumption that the system state is driven by finitely many Brownian motions with multiplicative noise, as described below:
\beq{x-spde} 
dX(t)= [\Dl X(t)+F(X(t))+u(t)]dt +\sum_{i=1}^N G_i(X(t))
dW^i(t)
\eeq
where $W^i=\qv{W,e_i}_H$ are the standard Brownian motions in $\rr$ and $G_i(X(t)):=G(X(t)) \CQ e_i$ with $\{e_i\}$ is an orthonormal basis of $H$. If we write $G(X(t))dW(t)$ as $\sum_{i=1}^\infty G(X(t))e_i dW^i(t)$, then the driving noise in \eqref{x-spde} can be viewed as a truncation of the noise $G(X(t))dW(t)$ in \eqref{sde}. The parameter $N$ represents the number of finitely many Brownian motions driving $X$. The case where the system is driven by an additive cylindrical Wiener process is simpler and can be understood analogously to a system driven by finitely many Brownian motions. 

Under this framework, the backward SPDE and SDE\eqref{BSPDE} take the form:
\beq{BSPDE-Q}\left\{\barray
dQ(t)\ad = -\Big[\Dl Q(t)+ \nabla_x^* F(X(t)) Q(t)+\nabla_x^* G_i(X(t)) q_1^i(t)\\
\aad \qquad + \nabla_x L(X(t),u(t))+\nabla_x^* h^k(X(t)) c_2^k(t)\Big]+ q_1^i(t) dW^i(t)+ q_2^k(t)dB^k(t) \\
Q(T)\ad = \nabla_x \Phi(X(T))
\earray\right.\eeq
and 
\beq{BSDE-C}\left\{\barray
dC(t) \ad = - L(X(t),u(t))dt + c_1^i(t)dW^i(t)+ c_2^k(t) dB^k(t) \\
C(T) \ad = \Phi(X(T))
\earray\right.\eeq
where $q_1^i :=q_1 e_i\in H$ and $c_1^i := c_1 e_i \in \rr$. In the above, we did not write out the summation $\sum_{i=1}^N$ as a convention.

To proceed, we discretize the forward-backward SPDEs \eqref{x-spde} and \eqref{BSPDE-Q} in space to obtain their numerical solutions. We note that \eqref{BSDE-C} is a backward stochastic differential equation, which does not involve spatial variables, and thus does not require spatial discretization.

\subsection{Spatial discretization for FB-SPDEs}
We consider the finite element space $S_h \subset H^1$ satisfying 
\bea
|P_h X- X|_H \to 0, \quad h \to 0
\eea
where $P_h: H \to S_h$ is the $L^2$-orthogonal projection onto $S_h$ defined by $\qv{P_h X-X, \phi_h} =0$ for all $\phi_h \in S_h$, and $\CR_h: H^1 \to S_h$ is the Ritz-orthogonal projection defined by $\qv{\nabla [\CR_h X -X], \nabla \phi_h}$=0 for all $\phi_h \in S_h$. The discrete Laplace operator $\Dl_h: S_h \to S_h$ is defined by $-\qv{\Dl_h X_h, \phi_h} =\qv{\nabla_h X_h, \nabla_h \phi_h}$ for all $\phi_h, X_h \in S_h$. We can introduce the following finite element partially observed control (FEPOC) problem. \newline

\noindent{\textbf{Problem (FEPOC):}} Let $h\in  (0,1)$. Minimize
\bea\ad 
J_h (u): = \EE^\QQ \Big[\int_0^T L(X_h(t), u_h(t)) dt \Big]+\EE \Phi(X_h(T)),
\eea
over the set $\CU_{ad}$ subject to the following partially-observed spatial-discretized SPDE:
\beq{dis-state}
\left\{\barray
d {X}_h(t) = [\Dl_h X_h(t)+ F(X_h(t))+ u_h(t) ] dt + G_i(X_h (t)) dW^i(t)\\
dY_h(t)= h(X_h(t))dt + dB(t),\\
X_h(0)= P_h X_0\in S_h, \, Y_0=0.
\earray\right.
\eeq
where $X_h=P_h X, u_h:= P_h u \in S_h$.

Let $\{Q_h\cd$, $q_{1,h}^i\cd, q_{2,h}^k\cd\}$ and $\{C_h\cd, c_{1,h}^i\cd, c_{2,h}^k\cd\}$ satisfy the following spatial-discretized backward SPDE and SDE, respectively,
\beq{fb-spde}\left\{\barray
d Q_h(t) \ad = -\Big[\Dl_h Q_h(t) + \nabla_x^* F(X_h(t)) Q_h(t)+ \nabla_x^* G_i(X_h(t)) q_{1,h}^i(t) \\
\aad \qquad +  \nabla_x L(X_h(t), u_h(t)) +  \nabla_x^* h(X_h(t)) c_{2,h}(t) \Big]dt \\
\aad \quad + q_{1,h}^i(t) dW^i(t) + q_{2,h}^k(t) dB^k(t), \\
d {C}_h(t) \ad = - L(X_h(t), u_h(t)) dt + c_{1,h}^i dW^i(t)+ c_{2,h}^k(t)dB^k(t) \\
Q_h(T) \ad = P_h[\nabla_x \Phi(X_h(T))],\quad  C_h(T)  = \Phi(X_h(T))
\earray\right. 
\eeq
where $q_{1,h}^i \in S_h$ and $q_{2,h}^k \in  S_h$ for each $i=1,\dots, N$ and $k=1,\dots, d$.  Here $\{C_h, c_{1,h}^i, c_{2,h}^k\}$ is the solution of backward SDE \eqref{BSDE-C} corresponding to spatial-discretized solution $X_h$ and $u_h$. The system \eqref{fb-spde} may be interpreted as the finite element discretization of first order optimality system of the partially observed stochastic optimal control problem \eqref{sde}-\eqref{Ju}. 

The following theorem asserts the convergence for the solution $(X_h, Q_h, q_{1,h}^i, q_{2,h}^k)$ toward $(X, Q, q_{1}^i , q_{2}^k)$. For this purpose, we consider the Banach space 
\bea
\CN[0,T]\ad := [L^2_\CF([0,T];V) \cap L^2(\Omega; C([0,T];H))]^2 \\
\aad\quad \times  L^2(\Omega \times [0,T]; \CL_2(H)) \times L^2(\Omega\times [0,T]; H),
\eea
endowed with the norm
\bea
\|(X,Q,q_1^i, q_2^k)\|_{\CN[0,T]}
\ad := \EE\Big[\sup_{t\in [0,T]} |X(t)|_H^2 + \sup_{t\in [0,T]} |Q(t)|_H^2 \Big] \\
\aad + \EE \int_0^T \Big( \|X(t)\|_V^2 + \| Q(t)\|_V^2+  \sum_{i=1}^N |q_1^i(t)|_{H}^2 + \sum_{k=1}^d |q_2^k (t)|_{H}^2 \Big) dt.
\eea

\begin{thm}\label{thm:conv}
Let $(X,Q, q_{1}^i, q_2^k)$ solve the forward-backward SPDE \eqref{sde}, \eqref{BSPDE-Q}, and $(X_h,Q_h,\\
q_{1,h}^i, q_{2,h}^k)$ be the solution of spatial discretized FB-SPDEs in \eqref{dis-state} and \eqref{fb-spde}. Then for fixed time horizon $T>0$, we have
\bea
\|(X,Q, q_{1}^i, q_2^k)- (X_h,Q_h, q_{1,h}^i, q_{2,h}^k) \|_{\CN[0,T]}^2 \to 0, \quad \text{ as } h \to 0.
\eea
\end{thm}

\begin{proof}
The proof of the convergence for finite element approximation of $(X,Q,q_1^i, q_2^k)$  in our problem can be found in \cite{Ben83}, thus we omit details. We also refer to the work of \cite{DP16} for convergence analysis of spatial discretization for forward-backward stochastic partial differential equations.
\end{proof}
\begin{rem}
With the convergence established in \thmref{thm:conv}, we can also establish the convergence for approximated cost $J_h\cd$ and $\{C_h\cd, c_{1,h}^i \cd, c_{2,h}^k \cd\}$ to $J\cd$ and $\{C\cd, c_1^i\cd, c_2^k\cd\}$ under their corresponding norms, respectively, as $h\to  0$.
\end{rem}

\subsection{Time discretization for spatial-discretized FBSPDE}\label{sec:num-FBSPDE}

To analyze the numerical approximation of the controlled spatial-discretized FBSPDE system \eqref{dis-state} and \eqref{fb-spde}, we consider its strong solution in the weak formulation. Suppose that $\Psi= \nabla_x^* \Phi (X_h(T))$, where the $S_h$-valued process ${X_h(t), t\in [0,T]}$ satisfies a spatially discretized forward SPDE driven by ${W^i}$. The corresponding time-discretized sequence is denoted by ${X_h(t_j): j=0, 1,\dots, N_T}$. To avoid the restrictive mesh constraint $\kappa \leq K h^2$, where $\kappa=t_{j+1}-t_j$ represents the uniform time step for the partition $\{t_j\}_{j=0}^{N_T}$ of $[0,T]$, we adopt an implicit Euler approximation scheme for both the forward and backward SPDEs.

Let us consider the controlled spatial-discretized FBSPDEs in \eqref{dis-state}-\eqref{fb-spde} on time interval $[t_j, t_{j+1}]$. For any $\phi_h\in S_h$, we have 
\beq{SDE}\barray
\qv{ X_h(t_{j+1}), \phi_h} \ad =\qv{X_h(t_j), \phi_h}-\int_{t_j}^{t_{j+1}} \qv{\nabla X_h(s), \nabla \phi_h} ds +\int_{t_j}^{t_{j+1}} \qv{F(X_h(s)), \phi_h} ds \\
\aad + \int_{t_j}^{t_{j+1}} \qv{ u_h(s), \phi_h} ds + \sum_{i=1}^N \int_{t_j}^{t_{j+1}} \qv{G_i(X_h(s)) dW^i(s), \phi_h},
\earray\eeq
and 
\beq{int-BSPDE-Q}\barray
\qv{Q_h(t_j),\phi_h} \ad = \qv{Q_h(t_{j+1}), \phi_h} -\int_{t_j}^{t_{j+1}} \qv{\nabla Q_h(s), \nabla \phi_h}ds \\
\aad + \int_{t_j}^{t_{j+1}}  \qv{\nabla_x^* F(X_h(s)) Q_h(s),\phi_h}ds  + \sum_{i=1}^N  \int_{t_j}^{t_{j+1}} \qv{\nabla_x^* G_i(X_h(s))q_{1,h}^i(s), \phi_h}ds \\
\aad + \int_{t_j}^{t_{j+1}}
\qv{\nabla_x L(X_h(s), u_h(s)), \phi_h}ds + \sum_{k=1}^d \int_{t_j}^{t_{j+1}} \qv{\nabla_x^* h^k(X_h(s)) c_{2,h}^k(s), \phi_h} ds \\
\aad - \sum_{i=1}^N \int_{t_j}^{t_{j+1}} \qv{q_{1,h}^i(s), \phi_h} dW^i(s) -\sum_{k=1}^d \int_{t_j}^{t_{j+1}} \qv{q_{2,h}^k(s), \phi_h} dB^k(s).
\earray\eeq
Similarly, for $C_h$ in \eqref{fb-spde}, we have 
\beq{int-BSDE-C}\barray
C_h(t_j)\ad = C_h(t_{j+1})+\int_{t_j}^{t_{j+1}} L(X_h(s), u_h(s))ds \\
\aad\quad -\sum_{i=1}^N \int_{t_j}^{t_{j+1}} c_{1,h}^i (s) dW^i (s) -\sum_{k=1}^d \int_{t_j}^{t_{j+1}} c_{2,h}^k(s)dB^k(s).
\earray\eeq
% Here $c_{2,h}=(c_{2,h}^1,\dots, c_{2,h}^d) \in \rr^d$ and $B=(B^1,\dots, B^d)$, where $B^i$ is $\rr$-valued standard Brownian motion. 
\beq{sim-X}\barray\ad 
\qv{X_h(t_{j+1}),\phi_h}+ \Dl t_j \qv{\nabla X_h(t_{j+1}), \nabla \phi_h} \\
\aad = \qv{X_h(t_j), \phi_h} + \qv{F(X_h(t_j)),\phi_h} \Dl t_j \\
\aad\quad + \qv{u_h(t_j), \phi_h} \Dl t_j + \sum_{i=1}^N \qv{G_i(X_h(t_j)),\phi_h}  \Dl_j W^i,
\earray\eeq
where $\Dl t_j= t_{j+1}-t_j$ and $\Dl_j W^i=W^i(t_{j+1})- W^i(t_j)$. The initial condition is $X_h(0)=P_h X_0$. 

For equation \eqref{int-BSPDE-Q}, which is a backward SDE after spatial discretization, we will obtain the following implicit Euler scheme: 
\begin{itemize}
\item[(i)] Set $Q_h(T)= P_h [\nabla_x^* \Phi(X_h(T))]$.
\item[(ii)] For every $j=N_T-1,N_T-2,\cdots,0$, simulate the $S_h$-valued random variables $q_{1,h}^{i}(t_j)$ and $Q_h(t_j)$ such that for all $\phi_h \in S_h$,
\beq{sim-q1} 
\qv{q_{1,h}^i(t_j),\phi_h}=\frac{1}{\Dl t_j} \EE \Big[\Dl_j W^i \bqv{Q_h(t_{j+1}), \phi_h} \Big| \CF_{t_j}\Big]
\eeq
where $\Dl_j W^i= W^i(t_{j+1})- W^i(t_j), i=1,\cdots,N$, and
\beq{sim-Q}\barray\ad 
\qv{Q_h(t_j), \phi_h}+ \Dl t_j \qv{\nabla Q_h(t_j), \nabla \phi_h} \\
\aad = \EE \Big[\qv{Q_h(t_{j+1}), \phi_h} + \Dl t_j \qv{\nabla_x^* F(X_h(t_j)) Q_h(t_{j+1}), \phi_h} \Big| \CF_{t_j}\Big]\\
\aad\qquad + \sum_{i=1}^N \Dl t_j \qv{\nabla_x^* G_i(X_h(t_j)) q_{1,h}^i(t_j), \phi_h} \\
\aad\qquad + \Dl t_j \qv{\nabla_x L(X_h(t_j), u_h(t_j)), \phi_h} + \sum_{k=1}^d \Dl t_j \qv{ \nabla_x^* h^k(X_h(t_j))c_{2,h}^k(t_j), \phi_h}.
\earray\eeq
\end{itemize}

We note that it is not necessary to compute $q_{2,h}^k$, which represents the martingale term of $Q_h$ with respect to Brownian motion $B^k$, $k=1,\dots, d$, since $q_{2,h}^k$ does not appear in the gradient \eqref{grad-J}. Equations \eqref{sim-q1} and \eqref{sim-Q} can be interpreted as projections of the solution onto the available information at each step while moving backward in time. For details on the Monte Carlo simulation of backward SDEs, we refer to the work of Bouchard and Touzi \cite{BT04} and references therein.

Similarly, to solve \eqref{int-BSDE-C} backward numerically, we will have 
\begin{itemize}
\item[(i)] Set $C_h(T)= \Phi(X_h(T))$.
\item[(ii)] For $j=N_T-1 , N_T-2, \cdots, 0$, simulate real-valued random variables  $c_{2,h}^k(t_j)$ and $C_h(t_j)$ such that 
\beq{sim-c2}
c_{2,h}^k (t_j)= \frac{1}{\Dl t_j} \EE \big[\Dl_j B^k C_h(t_{j+1}) | \CF_{t_j} \big],
\eeq
and 
\beq{sim-C}
C_h(t_j) = \EE[C_h(t_{j+1})| \CF_{t_j}] + \Dl t_j L(X_h(t_j), u_h(t_j)),
\eeq
where $\Dl_j B^k = (B^k(t_{j+1})- B^k(t_j))$ for $k=1,2,\cdots, d$.
\end{itemize}

We compute $c_{2,h}^k$ which represents the martingale term $C_h$ with respect to $B^k, k=1, \dots, d$, as it is required for calculating $Q_h(t_j)$ in \eqref{sim-Q}. However, computing $c_{1,h}^i, i=1,\dots, N,$ is not necessary.

For scheme \eqref{sim-X}-\eqref{sim-Q}, it is possible to reformulate the numerical scheme in algebraic form. For $\ell=1,\dots, \mathfrak{L}$, let $\phi_h^\ell \in S_h$ be the basis functions of $S_h$. Consider
\bea X_h(t_{j},\xi)\ad = \sum_{\ell=1}^\fL [\bold{X}_h(t_j)]_\ell \phi_h^\ell(\xi),\quad Q_h(t_j,\xi)= \sum_{\ell=1}^\fL [\bold{Q}_h(t_j)]_\ell \phi_h^\ell(\xi), \\
q_{1,h}^i(t_j,\xi)\ad =\sum_{\ell=1}^\fL [\bold{q}_{1,h}^i (t_j,\xi)]_\ell  \phi_h^\ell(\xi), \quad u_h(t_j,\xi) = \sum_{\ell=1}^\fL [\bold{u}_h(t_j)]_\ell \phi_h^\ell(\xi),
\eea
with coefficient vectors $\bold{X}_h(t_j), \bold{Q}_h(t_j), \bold{q}_{1,h}^i(t_j), \bold{u}_{h}(t_j) \in \rr^\fL$, where $[\cdot]_\ell$ denotes the $\ell$-th coordinate of the vector. 

In the following, we denote by $\bold{K}$ the stiffness matrix consisting of entries $\qv{\nabla \phi_h^\ell, \nabla \phi_h^w}$, where $\phi_h^\ell, \phi_h^w\in S_h$ are basis functions of $S_h$, while $\bold{M}$ (resp. $\bold{M}_{\nabla_x^* F}$ and $\bold{M}_{\nabla_x^* G_i}$) denote the mass matrices consisting of entries $\qv{\phi_h^\ell, \phi_h^w}$ (resp. $\qv{\nabla_x^* F(X_h(t_j)) \phi_h^\ell, \phi_h^w}$ and \newline
$\qv{\nabla_x^* G_i(X_h(t_j)) \phi_h^\ell, \phi_h^w}$). Then scheme \eqref{sim-X} can be reformulated as:
\beq{bold-X}\barray
(\bold{M}+\kappa \bold{K}) \bold{X}_h(t_{j+1}) \ad =\bold{M} \, \bold{X}_h(t_j)+ \kappa\, \bold{F}(t_j)  + \kappa\, \bold{M} \bold{u}_h(t_j) + \sum_{i=1}^N \bold{G}_i(t_j) \Dl_j W^i,
\earray\eeq
where $\bold{F}(t_j), \bold{G}_i(t_j) \in\rr^\fL$ with $[\bold{F}]_\ell := \qv{F(X_h(t_j)), \phi_h^\ell}$ and $[\bold{G}_i]_\ell := \qv{G_i(X_h(t_j)), \phi_h^\ell}$. Moreover, scheme \eqref{sim-q1} and \eqref{sim-Q} can be reformulated as follows:
\begin{itemize}
\item[(i)] Compute $\bold{M}\, \bold{Q}_h(T)= \nabla_x^* \bold{\Phi}_h(T)$, where $[\nabla_x^* \bold{\Phi}_h(T)]_\ell := \qv{\nabla_x^* \Phi(X_h(T)), \phi_h^\ell}$,
\item[(ii)] For $j=N_T-1,\dots, 0$, find the $\rr^\fL$-valued random variables $\bold{q}_{1,h}^i(t_j)$ and $\bold{Q}_h(t_j)$ such that 
\bea\ad 
\bold{M}\, \bold{q}_{1,h}^i(t_j)= \frac{1}{\kappa} \EE\big[\nabla_j W^i\, \bold{M}\, \bold{Q}_h (t_{j+1})| \CF_{t_j}\big],
\eea  
and 
\beq{bold-Q}\barray 
(\bold{M}+\kappa \bold{K}) \bold{Q}_h(t_j)\ad = \EE\big[\bold{M} \, \bold{Q}_h(t_{j+1})+ \kappa \bold{M}_{\nabla_x^* F} \bold{Q}_h(t_{j+1}) |\CF_{t_j}\big]\\
\aad\quad + \sum_{i=1}^N \kappa \bold{M}_{\nabla_x^* G_i} \bold{q}_{1,h}^i(t_j) +\kappa\, \nabla_x \bold{L}(t_j) + \sum_{k=1}^d \kappa\, \nabla_x^* \bold{h}^k \bold{c}_2^k (t_j),
\earray\eeq
where $\nabla_x \bold{L}(t_j) \in\rr^\fL$ with $[\nabla_x \bold{L}(t_j)]_\ell := \qv{\nabla_x L(X_h(t_j), u_h(t_j)), \phi_h^\ell}$ and $\nabla_x^* \bold{h}^k \bold{c}_2^k (t_j) \in \rr^\fL$ with $[\nabla_x^* \bold{h}^k \bold{c}_2^k (t_j)]_\ell := \qv{ \nabla_x^* h^k(X_h(t_j)) c_{2,h}^k(t_j), \phi_h^\ell}$. For notation $\nabla_x \bold{L}$, we keep $\nabla_x$ symbol in front of $\bold{L}$ to indicate that this term is calculated from $\nabla_x L(X_h(t_j), u_h(t_j))$. We adopt similar convention for notation $\nabla_x^* \bold{h}^k \bold{c}_2^k$.
\end{itemize}

\section{Numerical algorithm for solving partially observed optimal control}\label{sec:alg}
In this section, we introduce a general numerical algorithm for solving the partially observed optimal problem by combining the stochastic gradient algorithm, guided by the stochastic maximum principle in \thmref{thm:SMP}, with the particle filtering algorithm for computing nonlinear filtering.

\subsection{Update of the control}
For a prior chosen $\CF_t^Y$-adapted process $u_h^{0,Y}$, we determine the optimal control $u_h^{\star,Y}(t)$ at instant time $t\in [0,T]$ using gradient descend algorithm: 
\beq{GD}
u_h^{\iota+1,Y}(t)  = u_h^{\iota,Y}(t) -\al\, \nabla J(u_h^{\iota,Y}(t)), \quad \iota=0,1,2,\cdots
\eeq
where $\al$ is the step size, $\iota$ is the iteration index, and $u_h^{\iota,Y}(t)$ represents the finite element approximation of control $u^{\iota,Y}(t)$ at time $t$ and iteration number $\iota$.

From \eqref{nabla-J1}, computing  $\nabla J(u_h^{\iota,Y}(t))$ at instant time $t$ requires the trajectories $X_h(s), Q_h(s)$, $q_{1,h}^i(s)$, $C_h(s)$, and $c_{2,h}^k(s)$ of the spatial discretized forward-backward SPDE system \eqref{x-spde}-\eqref{BSDE-C} for $s\in [t,T]$, since $Q_h(s)$ and $C_h(s)$ are solved backward in time from $T$ to $t$. However, at instant time $t$, the observation information $\CF_s^Y$ for $s\in [t,T]$ is not yet available. This means that we cannot solve $\{X_h(s), Q_h(s), q_{1,h}^i(s), C_h(s), c_{2,h}^k(s)\}$ for $s \in [t,T]$ using estimated control $u_h^{\iota,Y}(s)$ at time $s$, as we only have access to $\CF_t^Y$ at time $t$.

A fundamental aspect of the partially observed control problem is that the control must be adapted to the observation filtration. Specifically, $u_h^{\iota,Y}(s)$ is $\CF_s^Y$-adapted for $t\leq s \leq T$. Yet at time $t$, the future observation filtration $\{\CF_{s}^Y\}_{t\leq s \leq T}$ is unknown. To address this challenge, we replace the control $u_h^{\iota,Y}(s)$ at time $s$ with its conditional expectation $\EE[u_h^{\iota,Y}(s)|\CF_t^Y]$. This substitution is justified by the fact that the conditional expectation provides the best approximation of $u_h^{\iota,Y}(s)$ given the available information $\CF_t^Y$. 

For $s\in [t,T]$, let us first take the conditional expectation with respect to $\CF_t^Y$ in the gradient descent algorithm \eqref{GD} over the interval $[t,T]$. It leads to 
\beq{cond-GD}
\EE\big[u_h^{\iota+1,Y}(s)|\CF_t^Y \big]= \EE\big[u_h^{\iota,Y }(s)| \CF_t^Y \big] -\al\, \EE \big[\nabla J(u_h^{\iota,Y}(s))|\CF_t^Y \big], \quad t\leq s \leq T,
\eeq
where $\EE[\nabla J(u_h^{\iota,Y}(s))|\CF_t^Y]$ is given by taking conditional expectation of \eqref{nabla-J1}: 
\beq{cond-grad}
\EE\big[\nabla J(u_h^{\iota,Y}(s))| \CF_t^Y \big] = \EE\big[\nabla_u^* F(X_h^{\iota,Y}(s), u_h^{\iota,Y}(s)) Q_h(s) + \nabla_u L(X_h^{\iota,Y}(s), u_h^{\iota,Y}(s)) |\CF_t^Y\big].
\eeq
We then replace $u_h^{\iota,Y}(s)$ in \eqref{cond-grad} by its conditional expectation $\EE[u_h^{\iota,Y}(s)|\CF_t^Y]$. 

Let us define 
\bea
u_h^{\iota,Y}(s)|_t:= \EE[u_h^{\iota,Y}(s)|\CF_t^Y].
\eea
Then for any $s\in [t,T]$, the conditional gradient descent algorithm \eqref{cond-GD} finally becomes
\beq{cond-GD1}
u_h^{\iota+1, Y}(s)|_t = u_h^{\iota,Y}(s)|_t -\al\, \EE\big[\nabla J(u_h^{\iota,Y}(s)|_t) | \CF_t^Y\big].
\eeq
We observe that when $s=t$,
\bea
u_h^{\iota,Y}(t)= \EE [u_h^{\iota,Y}(t)|\CF_t^Y],
\eea
because $u_h^{\iota,Y}(t)$ is adapted to observation filtration $\CF_t^Y$.

% Moreover, we note that \eqref{cond-grad} depends on $u_h^{\iota,Y}(s)$. However, as mentioned earlier, at time $t$, only information $\CF_t^Y$ is available. Therefore, $u_h^{\iota,Y}(s)$ should be replaced by $\EE[u_h^{\iota,Y}(s)| \CF_t^Y]$. 

In our work, we will develop numerical methods to compute the conditional estimated control process $\EE[u_h^{\iota,Y}(s)| \CF_t^Y]$ for $s\in [t,T]$ letting the time $t$ gradually increase. To proceed, we will apply the algorithm \eqref{cond-GD1} on a temporal grid:
\bea
\Pi:=\{t_n: 0 = t_0 < t_1 < t_2 <\cdots < t_{N-1} < t_N = T\},\;  n=0,1,\dots,N.
\eea
The algorithm \eqref{cond-GD1} at time $t_n$ reads as
\beq{cond-GD-ti}
u_{h}^{\iota+1,Y}(t_j)|_{t_n} = u_{h}^{\iota,Y}(t_j)|_{t_n} - \al\, \EE \big[ \nabla J(u_{h}^{\iota, Y}(t_j)|_{t_n})| \CF_{t_n}^Y\big],\quad  n \leq j \leq N.
\eeq
In the above,
\beq{E-grad}\barray\ad
\EE \big[\nabla J(u_{h}^{\iota,Y}(t_j)|_{t_n}) | \CF_{t_n}^Y \big] \\
\aad = \EE\big[\nabla_u^* F(X_{h}(t_j),  u_{h}^{\iota,Y}(t_j)|_{t_n}) Q_h(t_j) + \nabla_u L(X_h(t_j), u_{h}^{\iota, Y}(t_j)|_{t_n} ) \big| \CF_{t_n}^Y \big]\\
\aad = \int_{\rr^\fL} \EE\Big[ \psi(X_{h}(t_j),u_{h}^{\iota,Y}(t_j)|_{t_n}, Q_{h}(t_j)) \Big| X_h(t_n)=x \Big] \cdot p(x|\CF_{t_n}^{Y}) \, dx
\earray\eeq
where $\psi: \rr^\fL \times \rr^\fL \times \rr^\fL \to \rr$ is defined as 
\bea
\psi(X_h, u_h, Q_h):= \nabla_u^* F(X_h,u_h) Q_h + \nabla_u L(X_h, u_h),
\eea
and 
$p\left (\cdot |\CF_{t_n}^Y \right )$ is the probability density function (pdf) of the law of $X_{h}(t_n)$ given the observation information $\CF_{t_n}^Y$. This is where nonlinear filtering takes effect and plays a significant role in partially observed optimal control problems. To compute \eqref{E-grad}, the probability density function $p\left (X_{h}(t_n) |\CF_{t_n}^{Y} \right )$ will be approximated by its empirical distribution $\pi(X_{h}(t_n)|\CF_{t_n}^Y)$, which can be computed by the particle filtering algorithm, given by 
\bea\ad 
\pi(X_{h}(t_n)|\CF_{t_n}^Y) = \frac{1}{S} \sum_{s=1}^S \dl_{x_n^s} (X_{h}(t_n)).
\eea
where $\{x_n^s\}_{s=1}^S$ are the particle cloud at time $t_n$ and $S$ is the total number of particles. 

\subsection{Particle filtering for computing conditional distribution} \label{sec:particle}
The goal of nonlinear filtering is to determine the conditional distribution of the state process given observations.

In this paper, we adopt the Bayesian filter framework to approximate the filtering density $p(X_h(t_j)| \CF_{t_j}^Y)$. It consists of two steps: the prediction step and the update step. For the prediction step, suppose we know the distribution $p(X_{h}(t_{j-1})|\CF_{t_{j-1}}^Y)$ at time $t_{j-1}$. The prediction step then provides the pdf of thecontrolled process $X_h(t_j)$ at time $t_j$ given information $\CF_{t_{j-1}}^Y$ through the following Chapman-Kolmogorov equations:
\beq{pred}
p\big(X_h(t_j)|\CF_{t_{j-1}}^{Y}\big) = \int_{\rr^\fL} p\big(X_h(t_{j-1})|\CF_{t_{j-1}}^Y\big) p\big(X_h(t_j) | X_h(t_{j-1})\big) d X_h (t_{j-1}),
\eeq
where $p(X_h(t_j) |X_h(t_{j-1}))$ is the transition probability for the state process $X_h$ in \eqref{fb-spde} from $t_{j-1}$ to $t_j$. As the new observation data $Y(t_j)$ is received, the update step appplies the Bayesian inferences to update the prior probability density function and obtain the posterior pdf $p(X_h(t_j)|\CF_{t_j}^Y)$ as follows:
\beq{update}
p(X_h(t_j)| \CF_{t_j}^Y ) = \frac{p(X_h(t_j) | \CF_{t_{j-1}}^Y)\, p(Y(t_j)| X_h(t_j))}{ p(Y(t_j)| \CF_{t_{j-1}}^Y)},
\eeq
where $p(Y(t_j) | X_h(t_j))$ is the likelihood function that describes the discrepancy between the predicted state and the observations. To proceed, we introduce the bootstrap filter algorithm in  \cite{GSS93} as the benchmark particle filter algorithm, due to the efficiency for solving nonlinear filtering problems. 

At time $t_{j-1}$, suppose we have $S$ particles $\{x_{j-1}^s\}_{s=1}^S$ that follow the empirical distribution
\bea \ad 
\pi\Big(X_h(t_{j-1}) | \CF_{t_{j-1}}^Y \Big):= \frac{1}{S} \sum_{s=1}^S \dl_{x_{j-1}^s}(X_h(t_{j-1})).
\eea
It serves the approximation of the prior distribution $p(X_h(t_{j-1})|\CF_{t_{j-1}}^Y)$ at time $t_{j-1} $. Here $\dl_x$ represents the Dirac delta function at $x$. Therefore, the prior pdf in the prediction step \eqref{pred} can be approximated by 
\beq{prior-pdf}\barray\ad 
\wdt \pi \big(X_h(t_j) | \CF_{t_{j-1}}^Y \big)  := \frac{1}{S} \sum_{s=1}^S \dl_{\wdt x_j^s} ( X_h(t_j))
\earray\eeq
where $\wdt x_j^s$ are sampled from $\pi \big(X_h(t_{j-1})| \CF_{t_{j-1}}^Y \big) \, p\big(X_h(t_j) | X_h(t_{j-1}) \big)$. In other words, the empirical distribution of  sample cloud $\{\wdt x_j^s\}_{s=1}^S$ provides an approximation of the probability distribution $p(X_h(t_j)| \CF_{t_{j-1}}^Y)$. In the update step \eqref{update}, we replace $p\big(X_h(t_j)| \CF_{t_{j-1}}^Y \big)$ with $\wdt \pi\big(X_h(t_j) | \CF_{t_{j-1}}^Y \big)$ which is obtained in \eqref{prior-pdf} using particle cloud $\{\wdt x_j^s\}_{s=1}^S$. We then obtain the posterior pdf as:

\beq{posterior-pdf}\barray\ad 
\wdt \pi \big(X_h(t_j) | \CF_{t_j}^Y \big) := \frac{\sum_{s=1}^S \dl_{\wdt x_i^s}(X_h(t_j))  p(Y(t_j)| \wdt x_i^s)}{\sum_{s=1}^S p(Y(t_j) | \wdt x_i^s)} = \sum_{s=1}^S \omega_j^s \dl_{\wdt x_j^s}(X_h(t_j))
\earray\eeq
where the weights $\omega_j^s$ are proportional to $p\big(Y_h(t_j)| \wdt x_j^s \big)$. From this, we obtain a weighted empirical distribution $\wdt \pi(X_h(t_j) | \CF_{t_j}^Y)$ that approximates the posterior pdf $p(X_h(t_j)| \CF_{t_{j}}^Y)$ with importance density weight $\omega_j^s$ . In practice, after several time steps, the importance weights $\{\omega_j^s\}_{s=1}^S$ will tend to concentrate on a few samples, which reduces the effective particle size in the algorithm. To avoid the weight degeneracy problem, we resample particles $\{\wdt x_j^s\}_{s=1}^S$ by replacing particles with low density weights with copies of particles that have high weights. In the bootstrap particle filter \cite{GSS93}, the importance sampling method is used to generate equally weighted samples $\{x_j^s\}_{s=1}^S$ from $\wdt \pi\big(X_h (t_j)| \CF_{t_j}^Y\big)$. These resampled samples are then used to formulate the empirical distribution \beq{pf-pi} 
\pi \Big(X_h(t_j) | \CF_{t_j}^Y \Big) = \frac{1}{S} \sum_{s=1}^S \dl_{x_i^s}(X_h(t_j)).
\eeq
which serves as the approximation of $p\big(X_h(t_j)| \CF_{t_j}^Y\big)$.

\subsection{Stochastic gradient descent algorithm}
We now in the position to combine the numerical schemes for solving the forward-backward SPDEs in \secref{sec:num-FBSPDE}, the conditional gradient descent algorithm in \eqref{cond-GD-ti}, and the particle filtering algorithm from \secref{sec:particle} to formulate an efficient stochastic optimization algorithm to solve the partially observed optimal control problem.

Recall the conditional gradient descent algorithm in \eqref{cond-GD} requires to compute \eqref{E-grad}. 
We approximate the  conditional distribution $p\left (X_h(t_n) |\CF_{t_n}^Y\right )$ by the empirical distribution $\pi \left (X_h(t_n) | \CF_{t_n}^Y \right )$ from \eqref{pf-pi}. Therefore, \eqref{E-grad} can be approximated by
\beq{E-pi}\barray\ad 
\EE^\pi \big[\nabla J(u_h^{\iota,Y}(t_j)|_{t_n}) | \CF_{t_n}^Y\big]\\
\aad := \frac{1}{S} \sum_{s=1}^S \EE^\QQ \big[\nabla_u^* F(X_h(t_j), u_h^{\iota,Y}(t_j)|_{t_n}) Q_h(t_j)+ \nabla_u L(X_h(t_j), u_h^{\iota, Y}(t_j)|_{t_n}) | X_h(t_n)= x_n^s\big],
\earray\eeq
where $\{x_n^s\}_{s=1}^S$ are  samples from the distribution $\pi\left (X_h(t_n)| \CF_{t_n}^Y \right )$, which describes the approximated conditional pdf of the controlled process $X_h(t_n)$ given the observation information $\CF_{t_n}^Y$. Note that the right-hand side of \eqref{E-pi} is an expectation, and we can apply Monte Carlo simulation to compute this expectation. Specifically, $\EE^\pi[\nabla J(u_h^{\iota,Y}(t_j)|_{t_n}) | \CF_t^Y]$ can be approximated by 
\beq{Epi-MC}\barray\ad\!\!\!\! 
\EE^\pi[\nabla J(u_h^{\iota, Y} (t_j)|_{t_n}) | \CF_{t_n}^Y] \\
\aad\!\!\!\! \approx \frac{1}{S}\frac{1}{\La} \sum_{s=1}^S \sum_{\varsigma=1}^\La \Big[\nabla_u^* F(X_h^{\varsigma,s}(t_j), u_h^{\iota,Y}(t_j)|_{t_n}) Q_h^{\vsg,s}(t_j))+ \nabla_u L(X_h^{\vsg,s} (t_j), u_h^{\iota,Y}(t_j)|_{t_n}) | X_h(t_n)=x_n^s \Big]
\earray\eeq
where $\vsg$ is the index for $\La$ total numbers of samples used to approximate the expectation in the right-hand side of \eqref{E-pi}. The term $X_h^{\vsg,s}(t_j)$ is the $\vsg$-th realization of the controlled spatial-discretized state process with state $X_h(t_n)=x_n^s$, and $Q_h^{\vsg,s}(t_j)$ is the corresponding approximate solution $Q_h(t_j)$.

In the above computation, we observe that to approximate the conditional expectation in \eqref{Epi-MC}, we need $S\times \La$ samples of the  controlled state process to update a single gradient descent step in \eqref{cond-GD}. In practice, this becomes computationally expensive, especially in our case, where the dimension of the spatial-discretized state process $X_h(t_j)$ is high.

Motivated by the stochastic approximation algorithm, specifically the stochastic gradient descent algorithm, we aim to avoid approximating the conditional expectation in \eqref{Epi-MC} by all $S\times \La$ samples. Instead, we utilize a single realization to represent the conditional expectation; see \cite{ABYZ20}. This can be justified as follows: In the stochastic gradient descent algorithm (or stochastic approximation algorithm in general), the difference between the expectation and a single realization can be viewed as a mean-zero noise satisfying suitable properties. For instance, the noise in SGD can be of martingale difference type or even satisfy some mixing properties; see \cite{KY03}. Consequently, a single realization leads to
\bea
\nabla_u^* F(X_h^{\wdh \vsg, \wdh s}(t_j), u_h^{\iota,Y}(t_j)|_{t_n}) Q_h^{\wdh \vsg, \wdh s}(t_j) + \nabla_u L(X_h^{\wdh \vsg, \wdh s}(t_j), u_h^{\iota, Y}(t_j)|_{t_n})
\eea
where $X_h^{\wdh \vsg, \wdh s}(t_j)$ is a randomly generated realization of the controlled state process,  with initial state $X_h^{\wdh \vsg, \wdh s}(t_n) = x_n^{\wdh s}$ selected randomly from the particle cloud $\{x_n^{s}\}_{s=1}^S$. The term $Q_h^{\wdh \vsg, \wdh s}(t_j)$ represents the approximated solution of $Q_h(t_j)$ corresponding to this random sample $X^{\wdh \vsg, \wdh s}(t_j)$. For the  state-of-the-art study of the stochastic approximation algorithm, we refer the reader to the book of Kushner and Yin \cite{KY03}.

The conditional gradient descent algorithm in \eqref{cond-GD} then becomes the following conditional  stochastic gradient descent algorithm:
\beq{SGD}
u_h^{\iota,Y}(t_j)|_{t_n} = u_h^{\iota,Y}(t_j)|_{t_n} - \al\,\big[ \nabla_u^* F(X_h^{\wdh \vsg, \wdh s}(t_j), u_h^{\iota,Y}(t_j)|_{t_n}) Q_h^{\wdh \vsg, \wdh s}(t_j)+ \nabla_u L(X_h^{\wdh \vsg, \wdh s}(t_j), u_h^{\iota,Y}(t_j)|_{t_n}) \big]
\eeq
for $t_j \geq t_n, \iota=0,1,2,\dots,N_{\text{SGD}}$, where $N_{\text{SGD}}$ is the total number of iterations of SGD.

The algebraic form of algorithm \eqref{SGD} can be given as 
\beq{bold-SGD}
\bold{u}_h^{\iota,Y}(t_j)|_{t_n} = \bold{u}_h^{\iota,Y}(t_j)|_{t_n} - \al\,\big[ \bold{M}_{\nabla_u^* F}^{\iota,Y}(t_j)|_{t_n} \bold{Q}_{h}^{\wdh \vsg, \wdh s}(t_j)+ \nabla_u \bold{L}^{\iota,Y}(t_j)|_{t_n} \big]
\eeq
where $\bold{M}_{\nabla_u^* F}^{\iota,Y}(t_j)|_{t_n} \in \rr^{\fL\times \fL}$ with $\big[\bold{M}_{\nabla_u^* F}^{\iota,Y}(t_j)|_{t_n} \big]_{\ell,w} =\qv{\nabla_x^* F(X_h^{\wdh \vsg, \wdh s}(t_j), u_h^{\iota,Y}(t_j)|_{t_n}) \phi_h^\ell, \phi_h^w}$, and $\nabla_u \bold{L}^{\iota,Y}(t_j)|_{t_n}  \in \rr^\fL$ with $[\nabla_u \bold{L}^{\iota,Y}(t_j)|_{t_n}]_\ell=\qv{\nabla_u L(X_h^{\wdh \vsg, \wdh s}(t_j), u_h^{\iota,Y}(t_j)|_{t_n}), \phi_h^\ell}$. Here,  $\bold{Q}_h^{\wdh \vsg, \wdh s}(t_j)$ can be calculated by \eqref{bold-Q} corresponding to the state process $X_h^{\wdh \vsg, \wdh s}(t_j)$. 

From \eqref{bold-X}, $X_h^{\wdh \vsg, \wdh s}(t_j)$ can be computed by the following matrix form
\beq{Xh-hat}\barray\ad 
(\bold{M} + \kappa \, \bold{K}) \bold{X}_h^{\wdh \vsg, \wdh s}(t_{j+1})= \bold{M}\, \bold{X}_h^{\wdh \vsg, \wdh s}(t_j) + \kappa \bold{M} \, \bold{F}^{\wdh \vsg, \wdh s}(t_j)+ \kappa \bold{M}\, \bold{u}_h^{\iota, Y}(t_j)|_{t_n} + \sum_{i=1}^N \bold{G}_i^{\wdh\vsg, \wdh s}(t_j) \omega_i^{\wdh \vsg, \wdh s}
\earray\eeq
where $\bold{F}^{\wdh h, \wdh s}(t_j)$ and $\bold{G}_i^{\wdh \vsg, \wdh s}(t_j)$ are defined similar as \eqref{bold-X} by replacing $X_h(t_j)$ with $X_h^{\wdh \vsg, \wdh s}(t_j)$, initial $X_h^{\wdh \vsg, \wdh s}(t_n)=x_n^{\wdh s}\in \{x_n^s\}_{s=1}^S$, and $\omega_i^{\wdh \vsg, \wdh s} \backsim N(0,1)$. The $\{\omega_i^{\wdh \vsg, \wdh s}\}_{i=1}^{N_T-1}$ form a sequence of Gaussian random variables corresponding to the sample index $\wdh \vsg$ and $x_n^{\wdh s}$. 

Regarding the update of the control in the stochastic gradient descent algorithm \eqref{bold-SGD}, one needs to compute the value of $\bold{Q}_h^{\wdh \vsg, \wdh s}(t_j)$ corresponding to the particle $X_h^{\wdh \vsg, \wdh s}(t_n)=x_n^{\wdh s}$ at time $t_n$. The computation of $\bold{Q}_h^{\wdh \vsg, \wdh s}(t_j)$ in the algorithm \eqref{bold-Q} is still a Monte-Carlo type method, where expectation or conditional expectation is involved. To avoid directly calculating the conditional expectation,  we further utilize the idea of the  stochastic gradient descent algorithm,  representing the expectation by a single-realization of the trajectory. Therefore, from the solution path $\{\bold{X}_h^{\wdh \vsg, \wdh s}(t_j)\}_{j=n}^{N_T}$ using particles $x_n^{\wdh s}$ from \eqref{Xh-hat}, we obtain the numerical solution of $\bold{Q}_h^{\wdh \vsg, \wdh s}(t_j)$ and $\bold{q}_{1,h}^{i,\wdh \vsg, \wdh s}(t_j)$ as follows:
\beq{bold-q1h-hat}\barray\ad
\bold{M} \bold{q}_{1,h}^{i,\wdh \vsg, \wdh s}(t_j) = \frac{1}{\kappa} \bold{Q}_h^{\wdh \vsg, \wdh s}(t_{j+1})  \omega_i^{\wdh \vsg, \wdh s}(t_j)
\earray\eeq
and 
\beq{bold-Q-hat}\barray 
(\bold{M} + \kappa \bold{M}) \bold{Q}_h^{\wdh \vsg, \wdh s}(t_j) \ad = \bold{M} \bold{Q}_h^{\wdh \vsg, \wdh s}(t_{j+1})+ \kappa \bold{M}_{\nabla_x^* F} \bold{Q}_h^{\wdh \vsg, \wdh s}(t_j) \\
\aad + \sum_{i=1}^N \kappa \bold{M}_{\nabla_x^* G_i} \bold{q}_{1,h}^{i, \wdh \vsg, \wdh s}(t_j) + \kappa \nabla_x \bold{L}(t_j) + \sum_{k=1}^d \kappa \nabla_x^* \bold{h}^k\bold{c}_2^k (t_j)
\earray\eeq
Similarly, the computation of $c_{2,h}^k, k=1,\dots, d$ in \eqref{bold-Q-hat} will also be based on one single realization in numerical schemes \eqref{sim-c2} and \eqref{sim-C}. 

In summary, we integrate the particles in the particle filtering algorithm and the random samples in the gradient descent process into a unified stochastic gradient descent algorithm. This approach utilizes a single realization of a sample from  $\times \La$ calculations of $X_h^{\vsg, s}(t_j)$ and $Q_h^{ \vsg,s}(t_j)$ in \eqref{Epi-MC}. The resulting conditional optimal control process $\{\bold{u}_h^{N_{\text{SGD}},Y}(t_j)|_{t_n}\}$ from \eqref{bold-SGD} provides an estimate of the optimal control $\bold{u}_h^{\star,Y}(t_n)$ at time $t_n$ by 
\bea
\bold{u}_h^{\star,Y}(t_n):= \bold{u}_h^{N_{\text{SGD}},Y}(t_n)|_{t_n}.
\eea

We finally summarize our numerical algorithm for solving the partially observed optimal control problem of stochastic partial differential equations as follows:

\begin{Algorithm}{(FE-PF-SGD)}.\label{alg:FE-PF-SGD}
\begin{itemize}
    \item[(1)] Initialize the particle cloud $\{x_0^s\}_{s=1}^S \backsim \zeta$, where $\zeta$ is the initial distribution of $X_0$, and set the number of iteration $N_{\text{SGD}} \in \NN$.
    \item[(2)] Iterate time index $n=0,1,2,\dots, N_T$:
    \begin{itemize}
        \item[(i)]  Initialize the estimated control $\{u_h^{0,Y}(t_j)|t_n\}_{j=n}^{N_T}$ and set the learning rate $\al$.
        \item[(ii)] Iterate the SGD iteration $\iota=1,2,\dots, N_{\text{SGD}}$ for $N_{\text{SGD}}$ steps: 
        \begin{itemize}
            \item[(a)] \textbf{FSPDE}: compute one realization of controlled state process $\{X_h^{\wdh \vsg, \wdh s}(t_j)\}_{j=n}^{N_T-1}$ by \eqref{sim-X} with initial $X_h^{\wdh \vsg, \wdh s}(t_n)=x_n^{\wdh s}$ randomly selected from $\{x_n^s\}_{s=1}^S$.
            \item[(b)] \textbf{BSPDE}: compute one single realization of $\{ Q_h^{\wdh \vsg, \wdh s}(t_j), q_{1,h}^{i,\wdh \vsg, \wdh s} \}_{j=N_T}^n$ by \eqref{sim-q1} and \eqref{sim-Q} and one single realization of $\{C_h(t_j)$, $c_{2,h}^k(t_j)\}_{j=N_T}^n$ by \eqref{sim-C} and \eqref{sim-c2} corresponding to $\{ X_h^{\wdh \vsg, \wdh s}(t_j)\}_{j=n}^{N_T-1}$.
            \item[(c)] \textbf{SGD}: update the control using stochastic gradient descent algorithm \eqref{bold-SGD} to obtain $\{u_h^{\iota+1, Y}(t_j)|_{t_n}\}_{j=n}^{N_T}$. 
        \end{itemize}
    \item[(iii)] The estimated optimal control at instant time $t_n$ is given $u_h^{\star,Y}(t_n)= u_{h}^{N_{\text{SGD}},Y}(t_n)|_{t_n}$.
    \item[(iv)] Propagate particles using particle algorithm \eqref{prior-pdf}, \eqref{posterior-pdf}, and \eqref{pf-pi} to obtain particle cloud $\{x_{n+1}^s\}_{s=1}^S$ using optimal control $u^{\star,Y}(t_n)$ at instant $t_n$.
    \end{itemize} 
\end{itemize}
\end{Algorithm}

\section{Numerical Examples}\label{sec:num}
In this section, we present two numerical examples to demonstrate the effectiveness of our numerical algorithm for solving the partially observed optimal control problem. For spatial discretization, we consider Galerkin finite dimensional subspace
\bea\ad 
S_n : =\text{span}\Big\{1,\sqrt{\frac{2}{L}} \cos\Big(\frac{k}{L}\pi \cdot\Big) \Big| k=1,2,\dots, n \Big\},
\eea
with orthonormal basis
\bea\ad 
\phi_0 = \frac{1}{\sqrt{L}}, \; \phi_k = \sqrt{\frac{2}{L}} \cos\Big(\frac{k}{L} \pi \cdot \Big), \quad k=1,\dots, n,
\eea
or the finite element subspace 
\bea
\bar S_n = \text{span}\{\bar \phi_k | k = 1,2,\dots, n\} \subset H^1(0,L),
\eea
with basis
\bea
\bar\phi_k(\xi) = \left\{\barray
n\big(\xi - \frac{k-1}{n}L\big), \ad \quad\text{ if } \xi \in [(k-1)L/n, K L/n ], \\
n\big(\frac{k+1}{n}L - \xi \big), \ad \quad \text{ if } \xi \in [kL/n, (k+1)L/n ],  \\
0, \ad \quad \text{ otherwise}.
\earray\right.
\eea

In the following, we will always consider $n=400$ for finite element approximation. To solve the finite-element discretized version of \eqref{fb-spde} using the implicit Euler-Maruyama method, we take the time discretization step size $\Dl t= 0.01$ with terminal time $T=1$. The total iteration number of the stochastic gradient descent algorithm will be $N_{\text{SGD}} = 1,000$. The size of the  particle cloud will be $S=500$. The numerical algorithm is performed with Python on a Macbook Air equipped with Apple M2 chip, 16GB memory.

\begin{exm}{\rm{(Stochastic heat equation)}} \label{exm:heat}\rm{
We consider the following controlled stochastic heat equations with an additive cylindrical Wiener process:
\beq{heat}\left\{\barray
dX(t) =[\Dl X(t)+u(t)]dt + 0.05 dW(t) \quad t\in [0,1], \\
X(0) = X_0 \in H:=L^2(0,10),
\earray\right. \eeq
with $L=10$ and Dirichlet boundary conditions $X_0 = 0$. The observation process is taken as 
\beq{ob-heat}
dY(t)= h(X(t))dt+ dB(t), \quad Y(0)=0\in \rr^d.
\eeq
where $h(x)=\arctan(\qv{x,\sg_1}_H, \dots, \qv{x,\sg_d}_H)$ for some pre-selected $\sg_1, \cdots, \sg_d\in H$. In many practical applications, it is both natural and necessary to utilize a pre-selected set of elements \(\{\sg_i\}_{i=1}^d\) in the observation process. This is particularly relevant when dealing with complex system states, such as those arising in climate modeling, turbulence, and water movement in natural environments. Given the continuous and often high-dimensional nature of these systems, it is infeasible to observe their trajectories at every spatial point. Instead, observations must be strategically obtained from a finite number of monitoring locations. These pre-selected observation points serve as critical sources of information, enabling researchers to infer the underlying system dynamics while circumventing the impracticality of full-state measurement. The cost functional we consider is 
\bea\ad 
J(u)=\EE^\QQ \int_0^{1}\half \Big[ \|X(t)\|_{L^2(0,10)}^2 + \|u(t)\|_{L^2(0,10)}^2 \Big] dt + \half \EE^\QQ \|X(T)\|_{L^2(0,10)}.
\eea
In other words, in \eqref{sde}, we take \bea F(X(t),u(t))\ad =u(t),\quad G(X(t))=0.05, \\
L(X(t),u(t))\ad =\half \big[\|X(t)\|_{L^2(0,10)}^2 + \|u(t)\|_{L^2(0,10)}^2\big],\quad \Phi(X(T))=
\half \|X(T)\|_{L^2(0,10)}^2.
\eea
Therefore, $\nabla_x F=0, \nabla_u F=I$, $\nabla_x L(X(t),u(t))=X(t)$, $ \nabla_u L(X(t),u(t))=u(t)$, $\nabla_x \Phi(X(T))=X(T)$, and $\nabla_x^* h^k(X(t))c_2^k(t)= c_2^k(t) \sg_k /(1+\qv{X(t),\sg_k}_H^2)$. 

In numerical experiments, we take $d=3$. The preselected elements $\{\sg_i\}_{i=1}^d$ in $H$ will be projected to finite-dimensional space $S_h$, thus we have
$\sg_i \approx \sum_{\ell=1}^\fL \sg_{i,\ell}\phi_h^\ell$, where $\sg_{i,\ell}=\qv{\sg_i, \phi_h^\ell} \in \rr, i=1,\dots, d$. Define the matrix $\bold{\Sigma}:=(\sg_{i,\ell})_{i,\ell}\in \rr^{d\times \fL}$, we have $h(X_h(t))= \arctan(\bold{\Sigma} \bold{M} \bold{X}_h(t))$ and $ \sum_{k=1}^d \nabla_x^* \bold{h}^k \bold{c}_2^k(t)=\bold{M} \wdt{c}_2(t) \bold{\Sigma}$, where $\wdt c_2(t) = (\wdt{c}_2^1(t), \dots, \wdt{c}_2^d(t))$ with
\bea\ad 
\wdt{c}_2^k(t) = \frac{c_2^k(t)}{1+[\bold{\Sg} \bold{M} \bold{X}_h(t)]_k^2}.
\eea
Here $[\bold{\Sg} \bold{M} \bold{X}_h(t)]_k$ is the $k$-th coordinate of $\bold{\Sg} \bold{M} \bold{X}_h(t)$. For simplicity, we take $\bold{\Sigma}=\bold{I}\in \rr^{d\times \fL}$, the identity matrix with suitable dimensions. The numerical algorithm for solving backward stochastic partial differential equations \eqref{bold-Q} becomes 
\bea
(\bold{M}+\kappa \bold{K}) \bold{Q}_h(t_j)= \EE \big[\bold{M} \bold{Q}_h(t_{j+1})|\CF_{t_j}^Y\big]+\kappa \nabla_x \bold{ L}(t_j) + \sum_{k=1}^d \kappa \nabla_x^* \bold{h}^k \bold{c}_2^k (t_j).
\eea

After about $1,000$ iterations, we end up with an approximated cost of $J\approx 0.6327$. Below, we presented our simulation results. \figref{fig:heat-unc-state} displays one realization of the uncontrolled stochastic heat equation. \figref{fig:heat-control-state} presents our approximation of optimal control with partial observations (left) and the corresponding sample path of controlled solution (right).
\begin{figure}[H]
\includegraphics[width=0.5\textwidth]{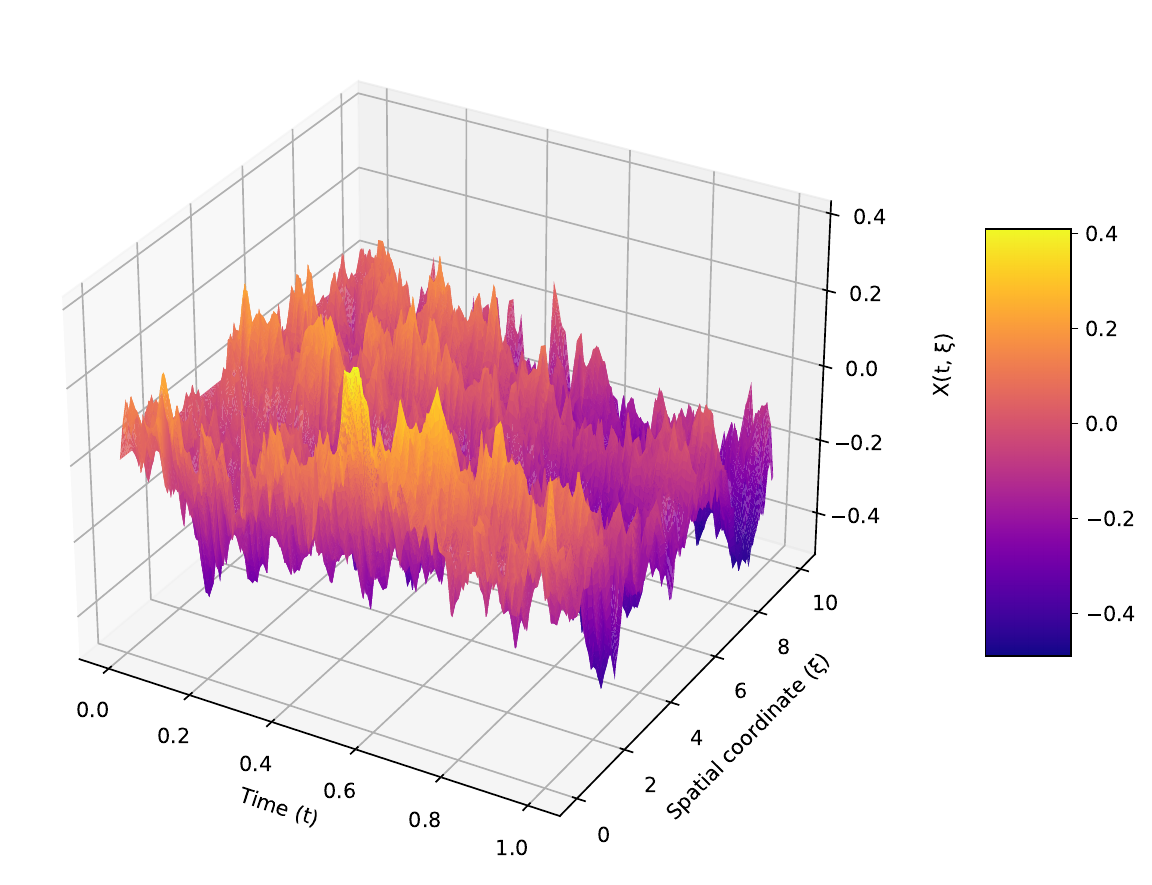}
\label{fig:heat-unc-state}
\caption{Sample path of stochastic heat equation without control.}
\end{figure}

\begin{figure}[H]
\includegraphics[width=0.48\textwidth]{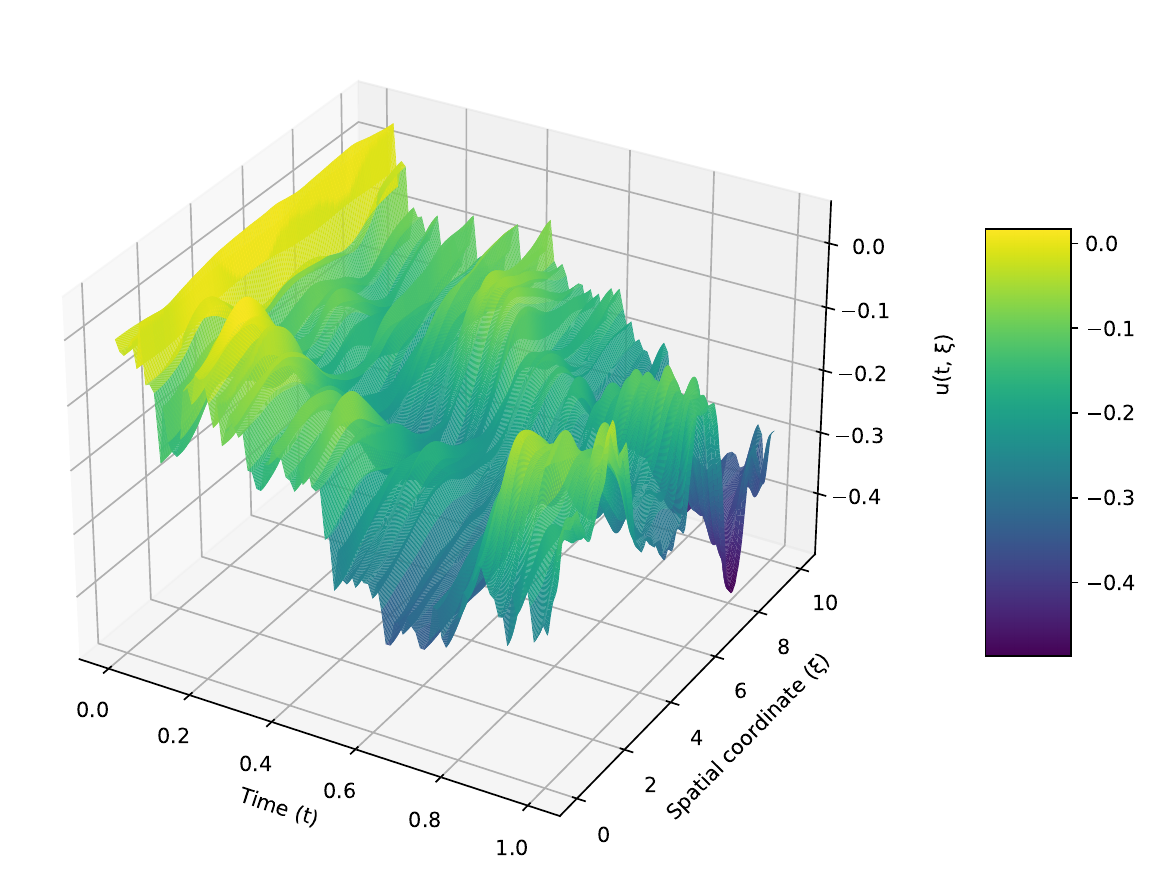}
\includegraphics[width=0.48\textwidth]{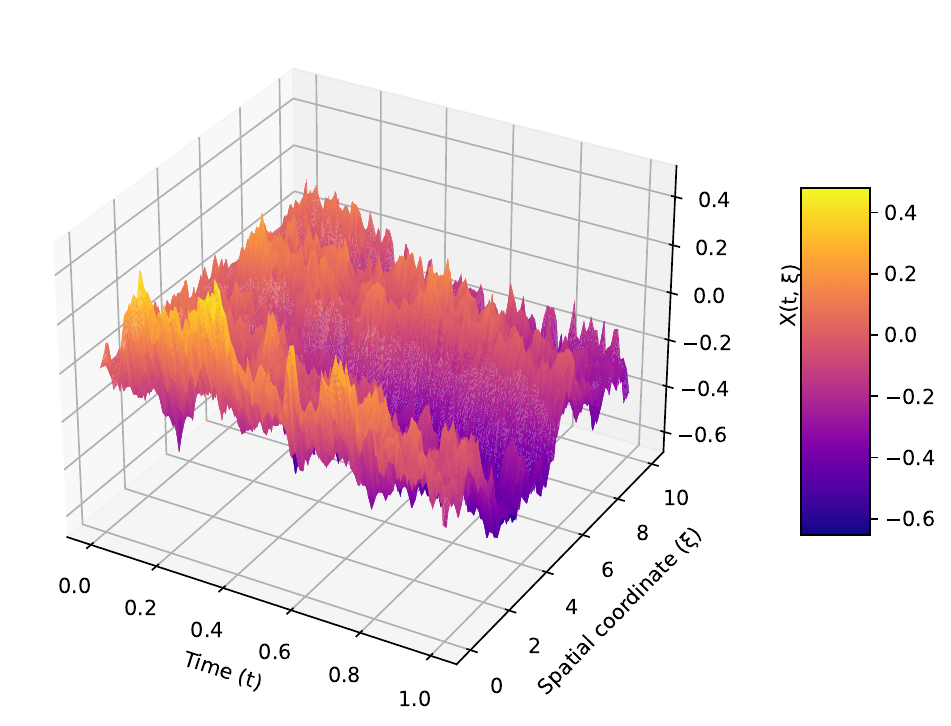} 
\caption{Sample path of approximated partially observed optimal control and the corresponding sample path of optimal controlled state solution.}
\label{fig:heat-control-state}
\end{figure}
}
\end{exm}

\begin{exm}{\rm{(Stochastic Nagumo equations)}} \label{exm:Nagu}\rm{
We consider the following partially observed controlled stochastic Nagumo equation:
\beq{Nagumo}\left\{\barray\disp
dX(t)= \big[\Dl X(t)- X(t)(X(t)-\frac{1}{2})(X(t)-1)+ u(t)\big]dt + \sum_{i=1}^N 0.05 (X(t)+1)e_i dW^i(t)\\
X(0) = X_0 \in L^2(0,20)
\earray\right.\eeq
with Neumann boundary condition $\partial X/\partial \bold{n}(t,\xi) = 0$ and $X_0\in \indi_{[5,15]}$. The observation process is taken as \eqref{ob-heat}. The setting of this example is not exactly the same as we discussed in our previous sections, since we deal with a zero Neumann boundary condition and dissipative drift with polynomial growth in \eqref{Nagumo}. However, the corresponding numerical algorithm takes the same form as in Algorithm \eqref{alg:FE-PF-SGD}. One needs to replace the finite element approximation by the spectral method using the Fourier transformation. The rest of the computations remain the same. We refer to \cite{FO16} and references therein for SMP of classical optimal control for SPDE with dissipative drift. We introduce the cost functional 
\beq{cost-nagu}
J(u)=\EE^\QQ \int_0^{T} \half \Big(\|X(t)-X^0(t)\|^2_{L^2(0,20)}+\|u(t)\|^2_{L^2(0,20)} \Big) dt +\half \EE^\QQ \|X(T)-X^0(T)\|^2, 
\eeq
where $X^0$ is the reference state, which is the solution of \eqref{Nagumo} without control and the noise.  

In the numerical experiment, we take $N=50, d=3$. After about $1,000$ iterations, we end up with an approximated cost of $J\approx 0.5536$. \figref{fig:Nagumo-unc-state} displays one realization of the reference state $X^0$ (left) and uncontrolled stochastic Nagumo equations (right).
\begin{figure}[H]
\includegraphics[width=0.48\textwidth]{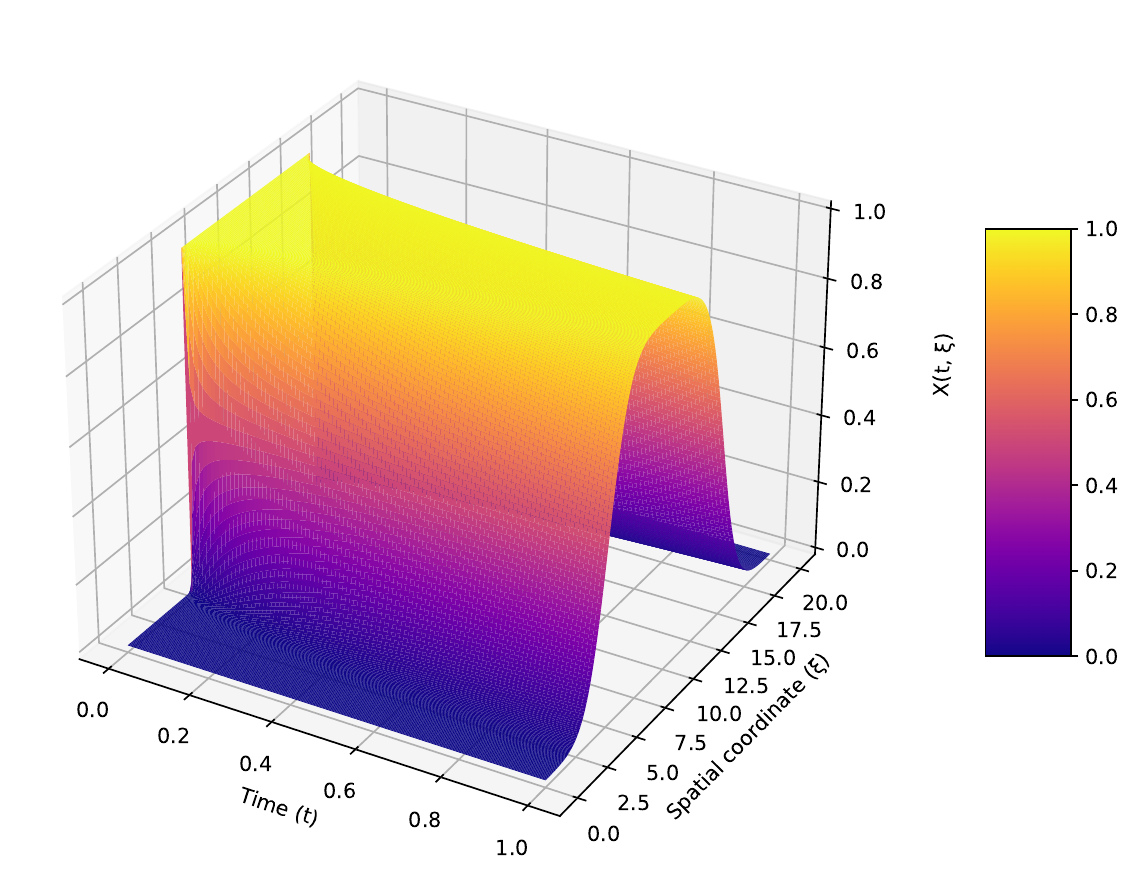}
\includegraphics[width=0.48\textwidth]{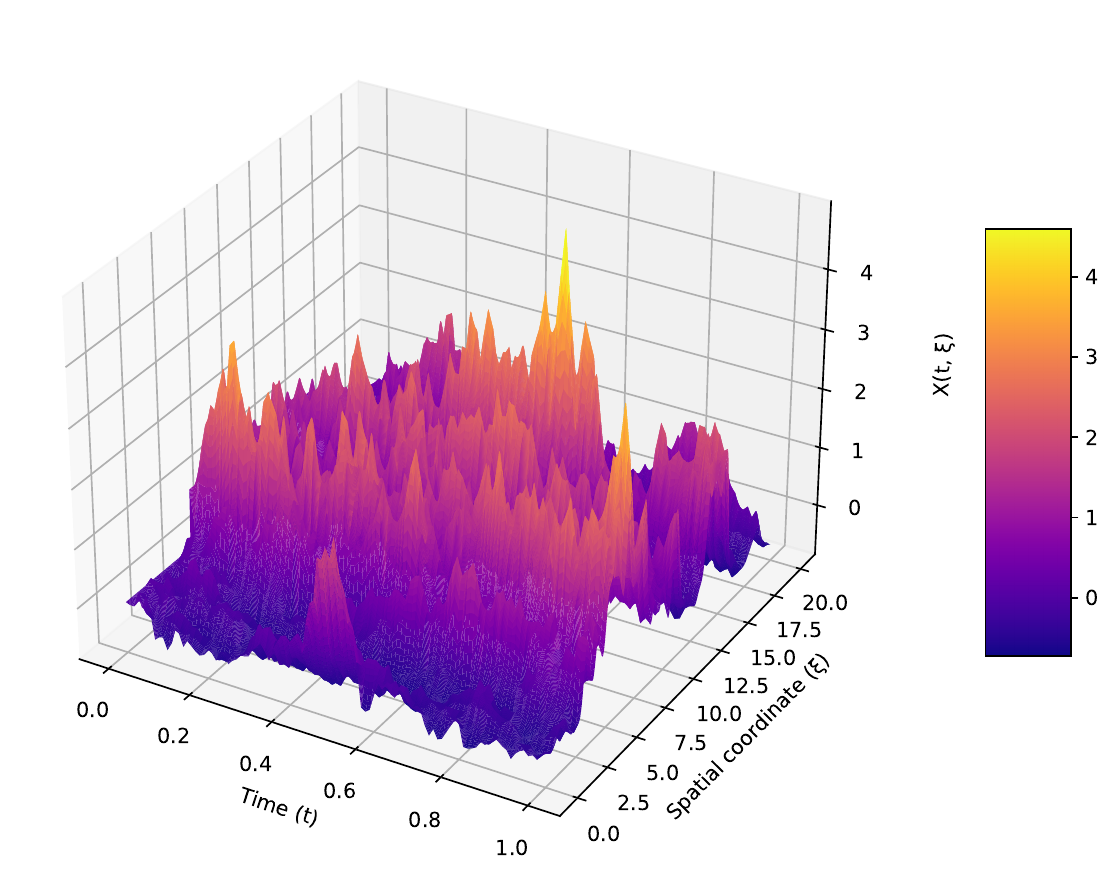}
\label{fig:Nagumo-unc-state}
\caption{Sample path of reference state and uncontrolled stochastic Nagumo equations.}
\end{figure}
\figref{fig:heat-control-state} presents our approximation of partially observed optimal control (left) and the corresponding sample path of solution of controlled stochastic Nagumo equations (right).
\begin{figure}[H]
\includegraphics[width=0.48\textwidth]{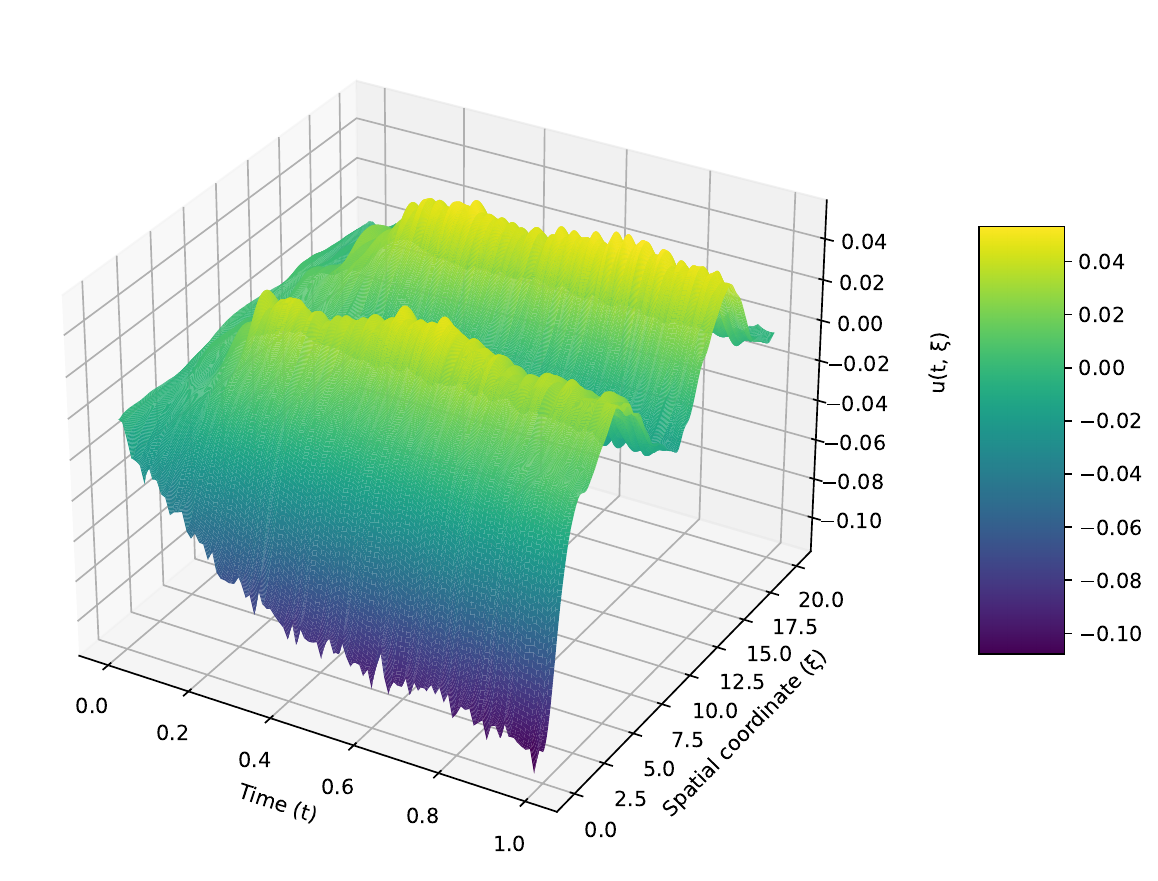}
\includegraphics[width=0.48\textwidth]{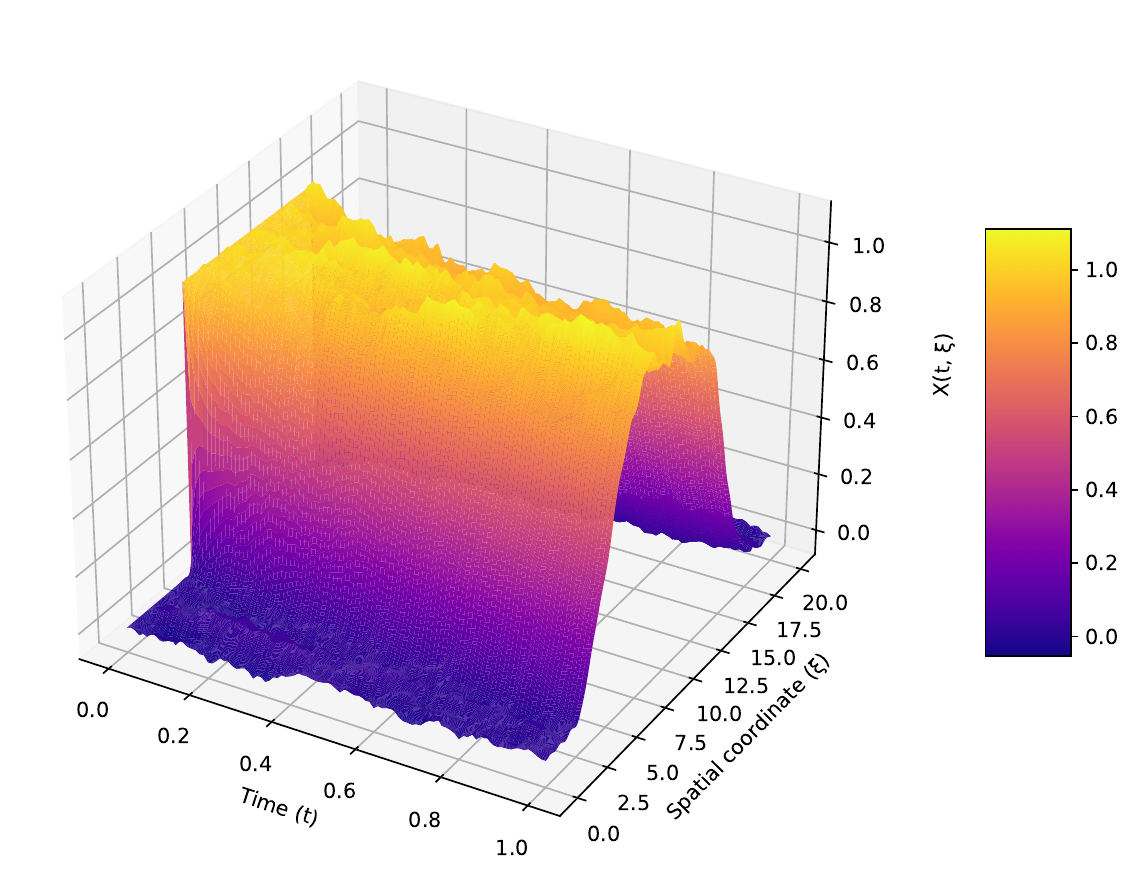}
\caption{Sample of approximated partially observed optimal control and the corresponding sample path of optimal controlled stochastic Nagumo equations.}
\label{fig:Nagumo-control-state}
\end{figure}
}
\end{exm}

\section{Concluding remarks} 
In this work, we established stochastic maximum principle for optimal control of stochastic partial differential equations with finite-dimensional partial observations, and then developed an efficient numerical framework for solving partially observed optimal control problems. Our approach integrates the numerical approximation for forward-backward SPDEs, particle filtering for state estimation, and stochastic gradient-based optimization. We use a single realization of the state and adjoint processes in stochastic gradient descent update to effectively balance computational efficiency and numerical accuracy, making it suitable for complex SPDE control problem. Future research directions include further theoretical analysis of the convergence properties of our stochastic gradient descent approach, as well as exploring its application to more complex SPDE models arising in filtering, finance, and fluid dynamics.

\section*{Acknowledgments}

F. Bao would  like to acknowledge the support from U.S. National Science Foundation through project DMS-2142672 and the support from the U.S. Department of Energy, Office of Science, Office of Advanced Scientific Computing Research, Applied Mathematics program under Grant DE-SC0025412. Y. Cao woud like to acknowledge the support from the U.S. Department of Energy, Office of Science, Office of Advanced Scientific Computing Research, Applied Mathematics program under Grants DE-SC0022253, DE-SC0025649. In addition, H. Qian would also like to thank Ruoyu Hu and Siming Liang for their valuable support on programming.

\end{document}